\documentclass[11pt,oneside,reqno]{amsart}
\usepackage{geometry}
\newcommand{\uA}{\underline{A}}
\newcommand{\QISG}{\mathrm{Q}\text{-}\mathrm{isog}}
\newcommand{\Isom}{\mathrm{Isom}}
\newcommand{\ldiv}{[\ell^\infty]}
\newcommand{\pdiv}{[p^\infty]}
\newcommand{\bfR}{\mathbf{R}}
\newcommand{\bfG}{\mathbf{G}}

\newcommand{\bfSp}{\mathbf{Sp}}
\newcommand{\bfGSp}{\mathbf{GSp}}

\newcommand{\bfPSp}{\mathbf{PSp}}

\newcommand{\Nm}{\mathrm{Nm}}
\newcommand{\Art}{\mathrm{Art}}
\newcommand{\Gal}{\mathrm{Gal}}
\newcommand{\GNrm}{\mathrm{GNrm}}

\newcommand{\etrm}{\mathrm{\acute{e}t}}

\newcommand{\degrm}{\mathrm{deg}}

\newcommand{\Trrm}{\mathrm{Tr}}

\newcommand{\ab}{\mathrm{ab}}
\newcommand{\GL}{\mathrm{GL}}

\newcommand{\Res}{\mathrm{Res}}

\newcommand{\ad}{\mathrm{ad}}

\newcommand{\proj}{\mathrm{pr}}

\newcommand{\ssrm}{\mathrm{ss}}
\newcommand{\Aut}{\mathrm{Aut}}
\newcommand{\End}{\mathrm{End}}

\newcommand{\crys}{\mathrm{crys}}

\newcommand{\Nrm}{\mathrm{Nrm}}
\newcommand{\CM}{\mathrm{CM}}
\newcommand{\A}{\mathbb{A}}

\newcommand{\C}{\mathbb{C}}

\newcommand{\F}{\mathbb{F}}
\newcommand{\G}{\mathbb{G}}

\newcommand{\N}{\mathbb{N}}

\newcommand{\Q}{\mathbb{Q}}
\newcommand{\R}{\mathbb{R}}

\newcommand{\Z}{\mathbb{Z}}

\newcommand{\Acal}{\mathcal{A}}

\newcommand{\Ccal}{\mathcal{C}}

\newcommand{\Gcal}{\mathcal{G}}
\newcommand{\Hcal}{\mathcal{H}}

\newcommand{\Ncal}{\mathcal{N}}
\newcommand{\Ocal}{\mathcal{O}}

\newcommand{\Rcal}{\mathcal{R}}
\newcommand{\Scal}{\mathcal{S}}
\newcommand{\Tcal}{\mathcal{T}}

\newcommand{\Gbf}{\mathbf{G}}
\newcommand{\Hbf}{\mathbf{H}}

\newcommand{\Pbf}{\mathbf{P}}

\newcommand{\Rbf}{\mathbf{R}}

\newcommand{\Ubf}{\mathbf{U}}

\newcommand{\hfrak}{\mathfrak{h}}
\newcommand{\lfrak}{\mathfrak{l}}
\newcommand{\pfrak}{\mathfrak{p}}

\newcommand{\Cfrak}{\mathfrak{C}}
\newcommand{\Dfrak}{\mathfrak{D}}

\newcommand{\Hfrak}{\mathfrak{H}}
\newcommand{\Lfrak}{\mathfrak{L}}
\newcommand{\Mfrak}{\mathfrak{M}}
\newcommand{\Nfrak}{\mathfrak{N}}
\newcommand{\Pfrak}{\mathfrak{P}}

\usepackage{imakeidx}
\makeindex[columns=3]
\makeindex[columns=2,options= -s index_style_xyz.ist]
\usepackage[utf8]{inputenc}
\usepackage{tgtermes}
\usepackage[english]{babel}
\usepackage{amsmath,
	mathtools,
	amssymb,
	graphicx,
	bbm,
	xcolor,
	amsthm,
	mathrsfs,
	imakeidx,
	mathdots}
\usepackage[all]{xy}
\usepackage[backref=page]{hyperref}
\hypersetup{}
\usepackage{hyperref}
\usepackage{tikz-cd}
\usetikzlibrary{decorations.pathmorphing}
\usepackage{dynkin-diagrams}
\usepackage{listings}
\usepackage{bm}
\usepackage[normalem]{ulem}
\usepackage{yhmath}
\makeatletter
\@namedef{subjclassname@2020}{%
	\textup{2020} Mathematics Subject Classification}
\makeatother

\theoremstyle{plain}

\newtheorem{theorem}{Theorem}[section]
\newtheorem{corollary}[theorem]{Corollary}
\newtheorem{definition}[theorem]{Definition}
\newtheorem{lemma}[theorem]{Lemma}
\newtheorem{proposition}[theorem]{Proposition}

\newtheorem{remark}[theorem]{Remark}

\newtheorem{example}[theorem]{Example}

\newtheorem{question}[theorem]{Question}

\newtheorem*{theorem-no-num}{Theorem}
\newtheorem*{proposition-no-num}{Proposition}
\newtheorem*{corollary-no-num}{Corollary}

\newtheorem{taggedtheoremx}{}
\newenvironment{taggedtheorem}[1]
{\renewcommand\thetaggedtheoremx{#1}\taggedtheoremx}
{\endtaggedtheoremx}

\usepackage[shortlabels]{enumitem}

\makeindex

\begin{document}
	\title{Simultaneous supersingular reductions of abelian varieties}
	\author{Xiaoyu Zhang}
	\address{Universit\"{a}t Duisburg-Essen,
		Fakult\"{a}t f\"{u}r Mathematik,
		Mathematikcarr\'{e}e
		Thea-Leymann-Straße 9,
		45127 Essen,
		Germany}
	\email{xiaoyu.zhang@uni-due.de}
	\date{\today}
	\maketitle

	\begin{abstract}
		
		For a point $x_0$ in a Shimura variety attached to a Shimura datum of Hodge type $(G,X)$, we have an associated abelian scheme $A_0$. Fixing a non-empty finite set $\Scal$ of primes, we consider the simultaneous supersingular  reduction modulo $\ell\in\mathcal{S}$ of (several copies of) $p$-adic Hecke orbits of $A_0$. We give a precise description of the image of this map. As an application, we give a more conceptual proof of Mazur's conjecture on non-torsionness of higher Heegner points on an abelian variety which is a quotient of the Jacobian of a Shimura curve. Our arguments simplify those of C.Cornut and V.Vatsal (\cite{Cornut2002,Vatsal2002,CornutVatsal2007}) in two important aspects: (1) we do not need to assume the $p$-adic group $G^1(\Q_p)/Z_{G^1(\Q_p)}$ to be simple; (2) we do not need to consider separately the ``geometric" part and ``chaotic" part in the Hecke orbits.
		
	\end{abstract}

	\tableofcontents

	\section{Introduction}
	B.Mazur conjectured in the Proceedings of 1983 ICM (\cite{Mazur1983}) that for any modular elliptic curve $E$ over $\Q$,\footnote{By the work of Breuil,Conrad,Diamond and Taylor (\cite{BreuilConradDiamondTaylor2001}), we know now that every elliptic curve over $\Q$ is modular.} there exists Heegner points of prime power conductors on $E(\overline{\Q})$ which are not torsion. This conjecture has deep applications in Iwasawa Main Conjecture via the construction of certain Heegner point Kolyvagin systems (\cite{Howard2004}).
	In this paper, we give a proof of a generalization of Mazur's conjecture for any abelian varieties that are quotients of Jacobians of Shimura curves. Our proof is much simpler and more conceptual than \cite{Cornut2002,Vatsal2002,CornutVatsal2007}. Such a result was used in \cite{Howard2004b} to prove one direction divisibility of Iwasawa Main Conjecture for such type abelian varieties.

	The previous proofs of Mazur's conjecture, as given in \cite{Cornut2002,Vatsal2002,CornutVatsal2007}, considered a \emph{geometric} part and a \emph{chaotic} part of the action of certain Galois group, the former of which involves a very intricate and careful analysis of Shimura curves of different levels. Our proof, however, is achieved via studying the image of simultaneous supersingular reductions of abelian varieties. We give a precise description of the image and Mazur's conjecture follows easily from this result. We do not need to consider these two parts as in \emph{loc.cit}. Instead, our description of the image already encodes both the \emph{geometric} part and the \emph{chaotic} part, and our result is more transparent and robust, which easily generalizes to Shimura varieties of Hodge type.

	As the main body of the article is about simultaneous supersingular reduction of abelian varieties and Mazur's conjecture is just an easy consequence of it, in this introduction, we will follow this line and first explain our main results on simultaneous supersingular reductions.

	\subsection{Simultaneous supersingular reductions of abelian varieties}
	Distributions of subvarieties in a Shimura variety is of fundamental importance in algebraic number theory, in particular in the study of arithmetic properties of Shimura varieties. Moreover, they have many surprising and interesting applications in André-Oort conjecture, Hecke orbit conjecture, Mazur's conjecture on Heegner points, etc.

	There is much research on the equidistribution of subvarieties in a Shimura variety $Sh$ in different settings: in the archimedean case $ Sh_{\mathbb{C}}$, \cite{ClozelUllmo2004a} considered strongly special subvarieties of positive dimension and proved a variant of André-Oort conjecture; \cite{Duke1988,EdixhovenYafaev2003,ClozelUllmo2004b,ShouwuZhang05} considered CM points and showed that a sequence of CM points which is `generic' in a certain sense is equidistributed. In the non-archimedean case $ Sh_{\mathbb{Q}_\ell}$, \cite{Disegni2022,GorenKassaei2021,HerreroMensaresRivera-Letelier2020,HerreroMensaresRivera-Letelier2021} considered CM points and studied the $\ell$-adic dynamics of Galois orbits/Hecke orbits of CM points. In the characteristic $\ell$ case $ Sh_{\overline{\mathbb{F}}_\ell}$, \cite{ChaiOort2019} considered an arbitrary point $x_0$ and showed that the Hecke orbit of such a point is Zariski dense in the central leaf of $ Sh_{\overline{\mathbb{F}}_\ell}$ containing $x_0$.

	In the following we write $ Sh_{k}$ for a Shimura variety of Hodge type over $k=\overline{\mathbb{F}}_\ell$ or $\mathbb{C}$.
	Fix a prime number $p\neq\mathrm{char}(k)$. The $p$-adic Hecke orbit $\mathcal{H}_p(x_0)$ of a point $x_0\in  Sh_{k}$ is the set of points $x\in  Sh_{k}$ related to $x$ by $p$-power isogenies. This notion is very important in the study of geometric properties of Shimura varieties in positive characteristic. In \cite{Cornut2002}, the author relates these two orbits from characteristic $0$ ($k=\mathbb{C}$) to characteristic $\ell$ ($k=\overline{\mathbb{F}}_\ell$) and studied the simultaneous reductions at several supersingular primes of Galois conjugates of an elliptic curve $E_0$ with CM. It is shown that as long as the Galois conjugates satisfy certain `rationality' condition, then this reduction map is always surjective onto the set of all supersingular elliptic curves. More refined results are obtained in \cite{CornutVatsal2007}.

	In this article, we study simultaneous supersingular reductions of (\emph{not necessarily CM}) abelian varieties at several different primes. Our main result gives a precise description of the image of this map under very mild conditions. Many of the results mentioned above in the archimedean case or non-archimedean case use techniques from ergodic theory, our work resonates well in this aspect. We should emphasize that in this article the abelian variety $A_0$ does not necessarily have CM. This case is not treated in previous works. However, our method and result illustrate that whether $A_0$ has CM or not, we can deal with the simultaneous reduction map on the same footing.

	CM points on modular curves and Shimura curve have indeed many deep applications, in particular in connection with Birch--Swinnerton-Dyer conjecture, Iwasawa Main Conjecture, Euler systems, Heegner points, Gross-Zagier formulas. There has been and continues to be an extensive research on CM points (see, for example, \cite{Howard2004,Vatsal2002,Nekovar2007,WeiZhang2014}, just to name a few references). Our result shows that the non-CM points on these curves contains also a lot arithmetic information and might therefore have many potential applications. Of course, when $A_0$ has CM, by the theory of complex multiplication, its Galois orbit is a subset of its toric orbit, and thus we recover the results in previous works for CM abelian variety $A_0$.

	Now let's introduce the main questions we are concerned with and the main results we prove in this article. We fix a finite set $\mathcal{S}$ of primes $\ell\neq p$ such that $x_0\in  Sh_\mathbb{C}$ has good reduction $x_{0}^{(\ell)}:=x_0(\mathrm{mod}\,\ell)$ at all $\ell\in\mathcal{S}$ and consider the distribution property of the simultaneous reduction $(x^{(\ell)})_{\ell\in\mathcal{S}}$ for $x\in\mathcal{H}_p(x_0)$, that is,
	
	\begin{question}\label{question-1}
		Is the following simultaneous reduction map surjective?
		\[
		\bfR_{\mathcal{S};p}
		\colon
		\mathcal{H}_p(x_0)
		\to
		\prod_{\ell\in\mathcal{S}}
		\mathcal{H}_p(x_{0}^{(\ell)}),
		\quad
		x
		\mapsto
		(x^{(\ell)})_{\ell\in\mathcal{S}}
		\]
	\end{question}
	In this article we restrict ourselves to Shimura varieties $ Sh_\mathbb{C}$ of Hodge type with $B=\mathbb{Q}$ and answer this question for a large class of points $x_0\in  Sh_\mathbb{C}$ (including CM points), assuming that the reduction $x_{0}^{(\ell)}$ is \emph{supersingular} for all $\ell\in\mathcal{S}$. In fact, for each $\ell\in\mathcal{S}$, we fix a finite set $\widetilde{\mathcal{T}}_\ell$ of elements in $\Gal((K')^\mathrm{ab}/K')$ where $K'$ is the reflex field of $x_0$. We give a complete answer to the following question which refines the preceding one ($\widetilde{\mathcal{T}}:=\{\widetilde{\mathcal{T}}_\ell\mid\ell\in\mathcal{S}\}$)
	\begin{question}\label{question-2}
		What is the image of the following simultaneous supersingular reduction map?
		\[
		\bfR_{\mathcal{S},\widetilde{\mathcal{T}};p}
		\colon
		\mathcal{H}_p(x_0)
		\mapsto
		\prod_{\ell\in\mathcal{S},\sigma\in\widetilde{\mathcal{T}}_{\ell}}
		\mathcal{H}_p(x_{0}^{(\ell)}),
		\quad
		x
		\mapsto
		((\sigma x)_\ell)_{\ell\in\mathcal{S},\sigma\in\widetilde{\mathcal{T}}_\ell}
		\]

		Fix a positive integer $N$ divisible by all $\ell\in\mathcal{S}$, we write $\mathcal{H}^{(N)}(x_0)$ for the prime-to-$N$ Hecke orbit of $x_0$. What is the image of the following simultaneous supersingular reduction map?
		\[
		\bfR_{\mathcal{S},\widetilde{\mathcal{T}}}^{(N)}
		\colon
		\mathcal{H}^{(N)}(x_0)
		\mapsto
		\prod_{\ell\in\mathcal{S},\sigma\in\widetilde{\mathcal{T}}_{\ell}}
		\mathcal{H}^{(N)}(x_{0}^{(\ell)}),
		\quad
		x
		\mapsto
		((\sigma x)_{\ell})_{\ell\in\mathcal{S},\sigma\in\widetilde{\mathcal{T}}}
		\]
	\end{question}
	These two questions have deep applications in Iwasawa theory, we refer the reader to §\ref{Mazur's conjecture on higher Heegner points} for more details.

	Now we introduce the main results of this article: we fix a non-degenerate symplectic pairing $\langle-,-\rangle\colon L\times L\to\Z$, write $\bfGSp_L$ for the similitude symplectic group scheme over $\Z$ associated to the symplectic pairing $\langle-,-\rangle$ and fix a compact open subgroup $U$ of $\bfGSp_L(\Z_p)$. We fix a polarized abelian variety with a $U$-level structure
	\[
	\uA_0
	=
	(A_0,\lambda_0,\overline{\eta}_0),
	\]
	where $A_0$ is an abelian variety over $K=\overline{\mathbb{Q}}$, $\lambda_0$ is a polarization of $A$ and $\overline{\eta}_0$ is a $U$-orbit of the identity map $\mathrm{Id}\colon A_0[p^\infty]\rightarrow A_0[p^\infty]$ where $A_0[p^\infty]$ is the $p$-divisible group associated to $A_0$ and we assume that the symplectic pairing $\langle-,-\rangle$ is compatible with the Weil pairing $\widehat{T}(A_0)\times\widehat{T}(A_0)\to\widehat{T}(\mathbb{G}_m)$ induced by $\lambda$ (here $\widehat{T}(A_0)$ is the $\widehat{\Z}$-Tate module associated to $A_0$). In other words, $\uA_0$ is an abelian variety with PEL structures and appears in the Shimura variety of Siegel type.
	We refer to §\ref{Polarized abelian varieties} for more details.

	We write $\bfG_{\uA_0}$ for certain group scheme associated to $\uA_0$ related to the units in the endomorphism ring $\mathrm{End}(\uA_0)$.
	We fix as above a finite set $\mathcal{S}$ of primes and for each $\ell\in\mathcal{S}$, fix a place $v(\ell)$ of $K$ over $\ell$ as well as a finite subset $\mathcal{T}_\ell$ of $Z_{\bfG_{\uA_0}}(\mathbb{Q}_p)$ where $Z_{\bfG_{\uA_0}}$ is the center of $\bfG_{\uA_0}$. We write $\uA^{(\ell)}$ for the reduction modulo $v(\ell)$ of $\uA$.

	We consider the $p$-adic Hecke orbit $\mathcal{H}_p(\uA_0)$ of $\uA_0$ (see Definition \ref{p-adic Hecke orbit}). Suppose for each $\ell\in\mathcal{S}$, $\uA_0$ has supersingular good reduction at $v(\ell)$, then so does every $\uA\in\mathcal{H}_p(\uA_0)$. The main result of this article is
	\begin{theorem}[Theorem \ref{R_{S,T} is an isomorphism}]\label{main result}
		Denote the image of each $\mathcal{T}_{\ell}$ in
		$Z_{\bfG_{\uA_0}}(\mathbb{Q})\backslash Z_{\bfG_{\uA_0}}(\mathbb{Q}_p)$ by $\overline{\mathcal{T}}(\ell)$. Then the image of $\bfR_{\mathcal{S},\mathcal{T};p}$ is in bijection with $\prod_{\ell\in\mathcal{S},\overline{\sigma}\in\overline{\mathcal{T}}(\ell)}
		\mathcal{H}_p(\uA_{0}^{(\ell)})$.

		In particular, suppose that for each $\ell\in\mathcal{S}$, the images of $\sigma\in\mathcal{T}_{\ell}$ in $Z_{\bfG_{\uA_0}^1}(\mathbb{Q})\backslash Z_{\bfG_{\uA_0}^1}(\mathbb{Q}_p)$ are all distinct, then $\bfR_{\mathcal{S},\mathcal{T};p}$ is surjective.
	\end{theorem}
	\begin{remark}
		The condition in the second part of the above theorem is the so-called `rationality' condition, which corresponds to the ``chaotic" part in the argument in \cite{Cornut2002,Vatsal2002,CornutVatsal2007}.

		The most significant difference between our result and the preceding results (for example \cite{ShouwuZhang05,CornutVatsal2007,Disegni2022}, etc.) is that $A_0$ can be a non-CM abelian variety of arbitrary dimension and the level subgroup $U\subset\mathbf{Sp}_L(\Z_p)$ can be arbitrarily small, and we give a precise description of the image of $\Rbf_{\Scal,\Tcal}$ without any constraints on $\Tcal$. Moreover, our proof is more streamlined and applies to Shimura datum of Hodge type. See the end of this introduction for the method that we use in this article to prove the above theorem.
	\end{remark}

	Our argument for Theorem \ref{main theorem-3} allows a generalization to Shimura varieties of Hodge type. Write $(G,X)$ for a Shimura datum of Hodge type and write $Sh_{U,\C}$ for the Shimura varieties over $\C$ of level $U\subset G(\A_f)$ attached to $(G,X)$. Similar to the map $\Rbf_{\Scal,\Tcal}$, we have a simultaneous supersingular reduction map
	\[
	\Rbf_{\Scal,\Tcal}^G\colon
	\Hcal_p^G(\Acal_{x_0})\to
	\prod_{\lfrak\in\Scal}\prod_{\sigma\in\Tcal_\lfrak}
	\Hcal_p^G(\Acal_{x_0^{(\lfrak)}}).
	\]
	Here $\Scal$ is a non-empty finite set of primes in $F$ and $\Tcal$ is as before.
	Then we have (we refer to §\ref{Shimura varieties of Hodge type} for unexplained notations below)
	\begin{theorem}[Theorem \ref{R_{S,T}, not simple}]\label{main theorem 4}
		We assume the following:
		\begin{enumerate}
			
			\item 
			(\ref{condition on similitude of I_x}) holds;

			\item 
			$G^1$ is simply connected;

			\item 
			there is a decomposition 
			$\Pbf G^1(\Q_p)=\prod_{j=1}^tG^{(j)}$ into subgroups with each $G^{(j)}$ simple and non-commutative;

			\item 
			for distinct $\lfrak,\lfrak'\in\Scal$, the subgroups $I_{x_0^{(\lfrak)}}(\Q)\bigcap U^p$ and $I_{x_0^{(\lfrak')}}(\Q)\bigcap U^p$ are \emph{not} commensurable;

			\item 
			for each $j=1,\cdots,t$ and $\lfrak\in\Scal$, $\Gamma_p(\lfrak)G^{(j)}$ is dense in $\Pbf G^1(\Q_p)$.
		\end{enumerate}
		Then the image $\Rbf_{\Scal,\Tcal}(\Hcal_p^G(\Acal_{x_0}))$ is equal to
		\[
		\prod_{\lfrak\in\Scal}
		\prod_{i=1}^{r(\lfrak)}
		\widetilde{\Delta}^{\Tcal_{\lfrak,i}}(\Hcal_p^G(\Acal_{x_0^{(\lfrak)}})).
		\]
	\end{theorem}
	From this, one deduces easily Mazur's conjecture (see Theorem \ref{main theorem 5}). Note here that we just assume $G^1$ simply connected and that $\Pbf G^1(\Q_p)$ is a product of simple and non-commutative groups, whereas in \cite{Cornut2002,Vatsal2002,CornutVatsal2007}, the authors need to assume that in addition that the group $\Pbf G^1(\Q_p)$ is \emph{itself} simple. This improves significantly the result in \cite{Cornut2002,Vatsal2002,Vatsal2002}. We refer to Remark \ref{compare CV and ours} for more on this.

	In the case where $A_0$ has CM by a CM field $L$ such that $\mathrm{End}_{K}(\uA_0)\otimes\mathbb{Q}\simeq L$, the above theorem is particularly interesting: write $L'$ for the reflex field of $L$ determined by the preceding isomorphism. Then we have group homomorphisms
	\[
	\Gal((L')^\mathrm{ab}/L')
	\xrightarrow{\Art_{L'}}
	\overline{(L')^\times}
	\backslash
	\A_{L',f}^\times
	\xrightarrow{\Nm}
	\bfG_{\uA_0}(\mathbb{Q})\backslash\bfG_{\uA_0}(\A_f)
	\]
	where $\Art_{L'}$ is the Artin reciprocity map and $\Nm\colon L'\rightarrow L$ is the reflex type norm.  For each prime $\ell\in\mathcal{S}$,
	fix a finite set $\widetilde{\mathcal{T}}_\ell$ of the preimage of $\bfG_{\uA_0}(\mathbb{Q})\backslash\bfG_{\uA_0}(\A_f)$ under the above map $\Nm\circ\Art_{L'}$.
	Again we have the simultaneous reduction map
	\[
	\bfR_{\mathcal{S},\widetilde{\mathcal{T}};p}
	\colon
	\mathcal{H}_p(\uA_0)
	\to
	\prod_{\ell\in\mathcal{S},\overline{\sigma}\in\overline{\mathcal{T}}(\ell)}
	\mathcal{H}_p(\uA_0),
	\quad
	\uA
	\mapsto
	((\sigma\uA)_{\ell})_{\ell\in\mathcal{S},\sigma\in\widetilde{\mathcal{T}}}.
	\]
	\begin{theorem}[Theorem \ref{surjectivity of Galois orbit}]\label{main theorem-3}
		Suppose that for each $\ell\in\mathcal{S}$, the images of the elements $\sigma\in\widetilde{\mathcal{T}}_\ell$ under the map
		\[
		\Gal((L')^\mathrm{ab}/L')
		\xrightarrow{\Nm\circ\Art_{L'}}
		\bfG_{\uA_0}(\mathbb{Q})\backslash\bfG_{\uA_0}(\A_f)
		\to
		\bfG_{\uA_0}(\mathbb{Q})\backslash\bfG_{\uA_0}(\mathbb{Q}_p)
		\]
		are all distinct, then $\bfR_{\mathcal{S},\widetilde{\mathcal{T}};p}$ is surjective.
	\end{theorem}

	\begin{remark}
		In the case where $A_0$ has CM by a CM field $L$, we can compare this theorem with \cite[Theorem 1.1]{AkaLuethiMichelWieser2022}. In \textit{loc.cit}, the authors considered a family of points of CM by (the maximal orders in) quadratic imaginary fields $\mathbb{Q}(-\sqrt{D_i})$ with $D_i\rightarrow\infty$ while here we consider a family of $p$-isogenies of a fixed CM point $A_0$ (thus a family of points of CM by orders of $p$-power discriminant in the fixed CM field $L$). See also \cite[Remark 1.4]{AkaLuethiMichelWieser2022}.

		Galois orbits of CM points in the Shimura variety $Sh$ contain very useful information and are related to many problems, for example, the subconvexity of special Rankin-Selberg $L$-values of modular forms (see for example \cite{Michel2004,ShouwuZhang05} and the references therein). Our Theorem \ref{surjectivity of Galois orbit} is much more general than Theorem \ref{main theorem-3}: we allow $A_0$ to be non-simple in Theorem \ref{surjectivity of Galois orbit} whereas in the above theorem we just state for $A_0$ simple. We expect our general result will have applications in the study of Hecke orbit conjectures, mod $\ell$ version André-Oort conjectures, etc.
	\end{remark}

	Putting together Theorem \ref{main result} for each prime $p\notin\mathcal{S}$, we get
	\begin{theorem}\label{main theorem for prime-to-N Hecke orbit}
		Fix a positive integer $N$ divisible by all $\ell\in\mathcal{S}$. Suppose for each $\ell\in\mathcal{S}$, any two distinct $\sigma_1,\sigma_2\in\mathcal{T}_\ell$, one has
		\[
		\sigma_1^{-1}\sigma_2\notin Z_{\bfG_{\uA_0}}(\mathbb{Q})\backslash Z_{\bfG_{\uA_0}}(\A_f)/\prod_{q\mid N}Z_{\bfG_{\uA_0}}(\mathbb{Q}_q),
		\] 
		then $\bfR_{\mathcal{S},\mathcal{T}}^{(N)}$ is surjective.
	\end{theorem}
	This is a special case of Theorem \ref{R_{S,T}^N is an isomorphism}, which answers Question \ref{question-2}.

	We outline the proof of Theorem \ref{main result} for the case that $\mathcal{T}_\ell$ is in bijection with $\overline{\mathcal{T}}_\ell$ for all $\ell\in\mathcal{S}$. This relies on an explicit description of the $p$-adic Hecke orbit of $\mathcal{H}_p(\uA_0)$ and $\mathcal{H}_p(\uA_{0}^{(\ell)})$. In fact, we have bijections (see §\ref{p-adic Hecke orbit of A}, in particular Theorem \ref{description of H_p(A_0)}, (\ref{p-adic Hecke orbit of A, char=0}) and (\ref{p-adic Hecke orbit of A, char=ell})):
	\begin{align*}
		\mathcal{H}_p(\uA_0)
		&
		\simeq
		\bfG_{\uA_0}(\Z[1/p])
		\backslash
		\bfGSp_L(\mathbb{Q}_p)/U,
		\\
		\mathcal{H}_p(\uA_{0}^{(\ell)})
		&
		\simeq
		\bfG_{\uA_{0}^{(\ell)}}(\Z[1/p])
		\backslash
		\bfGSp_L(\mathbb{Q}_p)/U.
	\end{align*}
	Such descriptions, even though very simple, are all obtained in previous works ( for example \cite{CornutVatsal2007,AkaLuethiMichelWieser2022}) via the theory of Deuring on complex multiplication. Our approach, inspired by \cite{Yu2010}, is much easier and much more transparent and we obtain more general results than those in the literature: here $U$ can be arbitrarily small and $A_0$ needs not to have CM while previous works always assume $U$ to contain $\Gamma_0(N)$ and $A_0$ to be CM.

	So Theorem \ref{main result} is reduced to the following
	\begin{theorem}[Theorem \ref{surjectivity of Pi_R}]\label{main theorem-2}
		Assume that $\mathcal{T}_\ell$ is in bijection with $\overline{\mathcal{T}}_\ell$ for all $\ell\in\mathcal{S}$. Then the following map is surjective
		\[
		\bfSp_L(\mathbb{Q}_p)
		\to
		\prod_{\ell\in\mathcal{S},\sigma\in\mathcal{T}_{\ell}}
		\bfG_{\uA_{0}^{(\ell)}}^1(\Z[1/p])
		\backslash
		\bfSp_L(\mathbb{Q}_p)/U^1,
		\quad
		g
		\mapsto
		\left(
		\bfG_{\uA_{0}^{(\ell)}}^1(\Z[1/p])\sigma gU
		\right)_{\ell\in\mathcal{S},\sigma\in\mathcal{T}_\ell},
		\]
		where $U^1=U\bigcap\bfSp_L(\Q_p)$.
	\end{theorem}
	The proof of the theorem relies on a careful study of commensurators of $\bfG_{\uA_{0}^{(\ell)}}(\Z[1/p])$ inside $\bfGSp_L(\mathbb{Q}_p)$ and Ratner's theorems on unipotent flows. We follow the strategy in \cite{Cornut2002}. The main difficulty is to show that for $\sigma_1,\sigma_2\in\mathcal{T}_\ell$, $\sigma_1\sigma_2^{-1}\notin Z_{\mathbf{G}_{\underline{A}_0}}(\mathbb{Q})$ if and only if $\sigma_1^{-1}\mathbf{G}_{\underline{A}_0}(\Z[1/p])\sigma_1$ and $\sigma_2^{-1}\mathbf{G}_{\underline{A}_0}(\Z[1/p])\sigma_2$ are not commensurable (Theorem \ref{commensurable}).

	\begin{remark}\label{compare CV and ours}
		The proof of the above theorem generalizes to arbitrary connected reductive group $G/\Q$ whose derived subgroup $G^1$ is simply connected. The relation between our proof and \cite{Cornut2002,Vatsal2002,CornutVatsal2007} is as follows: in \cite{Cornut2002,Vatsal2002}, the key point is that the group $\mathrm{PSL}_2(\Q_p)$ is a simple and non-commutative group. By a very interesting characterization of subgroups of a finite product of $\mathrm{PSL}_2(\Q_p)$ which is normalized by the diagonal image of $\mathrm{PSL}_2(\Q_p)$ (see also Proposition \ref{product of diagonals}), the authors in \cite{Cornut2002,Vatsal2002} deduce the surjectivity of the map $\Rbf_{\Scal,\Tcal}$ under the rationality condition. In our proof, we give a characterization of similar subgroups of finite product of non-commutative group $\Pbf G^1(\Q_p)=G^1(\Q_p)/Z_{G^1(\Q_p)}$ (see also Proposition \ref{H is a product of diagonals}). Combined with our assumption that $G^1$ is simply connected, we can apply the strong approximation property to $\Pbf G^1(\Q_p)$ and prove the above theorem. We refer to §\ref{G not simple} and §\ref{Shimura varieties of Hodge type} for more details.
	\end{remark}

	\subsection{Mazur's conjecture on higher Heegner points}\label{Mazur's conjecture on higher Heegner points}
	One motivation of this article comes from generalizing \cite{Cornut2002} to Shimura varieties of Hodge type. Another motivation comes from	reproving Mazur's conjecture in a more natural way than in \cite{Cornut2002,Vatsal2002,CornutVatsal2007}. In \emph{loc.cit}, the authors considered two parts of the action of certain Galois group and the \emph{geometric} part involves a delicate analysis of certain auxiliary Shimura curves.

	Let's first recall Mazur's conjecture. We fix a CM number field $E$ with maximal totally real subfield $F$. We fix one archimedean place $\xi\colon F\hookrightarrow\R$ of $F$. We also fix an ideal $\Nfrak$ of $\Ocal_F$ coprime to the relative discriminant $E/F$. We assume that the weak Heegner hypothesis is satisfied: $\epsilon_{E/F}(\Nfrak)=(-1)^{[E:F]-1}$. Then there is a unique quaternion algebra $B$ over $F$ that is ramified exactly at the primes $\Pfrak$ of $F$ inert in $E$ such that $\mathrm{ord}_{\Pfrak}(\Nfrak)$ is odd and at all archimedean places of $F$ except $\xi$. For each compact open subgroup $U$, one associates a Shimura curve $Sh_{U,\C}$ over $\C$ of level $U$.

	One can construct a family of special points, called Heegner points $x[\lfrak^k]$ on $Sh_{U,\C}$ for each prime $\lfrak$ of $\Ocal_F$ coprime to $\Nfrak D_{E/F}$ and $k\ge0$. We fix one such prime $\Pfrak$ (over the rational prime $p$), then all these $x[\pfrak^k]$ ($k\ge0$) are defined over $E[\pfrak^\infty]=\lim\limits_{\overrightarrow{k}}E[\pfrak^k]$ with $E[\pfrak^k]$ the ring class field of $E$ corresponding to $h[\pfrak^k]$. Write $H_\infty$ for the maximal anti-cyclotomic $\Z_p^{[F_\pfrak:\Q_p]}$-extension, then $E[\pfrak^\infty]/H_\infty$ is a finite extension. For any abelian variety $A$ over $F$ with a surjective morphism $Sh_{U,F}\to A$, we write $y[\pfrak^k]$ for the image of $x[\pfrak^k]$. Then we have the following theorem (Theorem \ref{Mazur's conjecture}), which is a conjecture of Mazur (\cite[p.203]{Mazur1983})
	\begin{theorem}\label{main theorem 5}
		There exists $k\ge0$ such that
		\[
		\Trrm_{E[\pfrak^\infty]/H_\infty}(y[\pfrak^k])\notin A(H_\infty)_\mathrm{tors}.
		\]
	\end{theorem}
	Mazur's conjecture is inspired by the famous Gross-Zagier formula, the latter implies in particular that $\Trrm_{E[1]/E}(y[1])\notin A(E)_\mathrm{tors}$ if and only if $L'(A/E,1)\neq0$. This conjecture was later applied to Iwasawa theory. We refer to \cite{Cornut2002,Howard2004,Howard2004b} for some of the applications of this conjecture in this direction. The above theorem is an easy consequence of Theorem \ref{main theorem 4}.

	According to \cite{Cornut2002,Vatsal2002,CornutVatsal2007}, the action of the Galois group $G_0:=\Gal(E[\pfrak^\infty]/H_\infty)$ can be divided into two parts: the \emph{geometric} part corresponds to the subgroup $G_1\subset G_0$ generated by the Frobenius in $\Gal(E[\pfrak^\infty]/E)$ of ramified primes in $E/F$ prime to $\pfrak$; the \emph{chaotic} part corresponds to the quotient group $G_0/G_1$. Fix a set $\Rcal$ of representatives in $G_0$ of the quotient $G_0/G_1$. Then we can rewrite
	\[
	\Trrm_{E[\pfrak^\infty]/H_\infty}(y[\pfrak^k])=\sum_{\sigma\in\Rcal}\sum_{g\in G_1}\sigma(g(y[\pfrak^k])).
	\]
	The map $y[\pfrak^k]\mapsto\sum_{g\in G_1}g(y[\pfrak^k])$ comes from a morphism of modular curves
	$Sh_{\Nfrak\Mfrak,\C}\to(Sh_{\Nfrak,\C})^{2^{[F:\Q]}}$ where $\Mfrak=\lfrak_1\cdots \lfrak_{[F:\Q]}$ is the product of ramified primes $\lfrak_i$ in $E/F$ prime to $\pfrak$ (see \cite[§4.1]{Cornut2002} for more details).

	Our proof of Mazur's conjecture, on the other hand, treats the geometric part $G_1$ and the chaotic part $\Rcal$ in a uniform manner and is encoded in the description of the image of the map $\Rbf_{\Scal,\Tcal}$.
	As we have given enough details in the previous subsection of this introduction, we refer to the proof of Theorem \ref{Mazur's conjecture} for more details.

	\subsection*{Outline}
	Here is an outline of this article: in §\ref{Polarized abelian varieties}, we introduce the main objects of study: polarized abelian varieties. In §\ref{p-adic Hecke orbit of A} and §\ref{prime-to-N Hecke orbits of A_0}, we give explicit descriptions of the $p$-adic/prime-to-$N$ Hecke orbits of $\uA_0$ and reduce Theorem \ref{main result} to Theorem \ref{main theorem-2}. In §\ref{Commensurability criterion}, we study the commensurators of certain arithmetic subgroups of $\bfGSp_L(\mathbb{Q}_p)$. In §\ref{Proof of the main result}, we prove Theorem \ref{main theorem-2}. In §\ref{Shimura varieties of Hodge type}, we generalize our results to Shimura varieties of Hodge type.	
	In §\ref{CM abelian varieties}, we apply Theorem \ref{main result} to the case where $A_0$ is a CM abelian variety. In §\ref{Heegner points and Mazur's conjecture}, we apply our results in §\ref{Shimura varieties of Hodge type} to Shimura curves and reprove Mazur's conjecture.

	\subsection*{Notations}
	For a positive integer $r$, we write
	\[
	\Z_{(r)}=\Z[\frac{1}{r}],
	\quad
	\Z^{(r)}=\Z[\frac{1}{\ell}|\,\forall\,\ell\nmid r].
	\]
	\[
	\A_r=\prod_{\ell\mid r}\mathbb{Q}_\ell,
	\quad
	\A^{(r)}_f=\prod_{\ell\nmid r}\,'\mathbb{Q}_\ell.
	\]
	Here $\prod\,'$ is the restricted product. Similarly $\widehat{\Z}^{(r)}=\prod_{\ell\nmid r}\Z_\ell$.
	For a finite set of primes $\mathcal{N}=\{p_1,\cdots,p_s\}$, we write
	\[
	\Z_{(\mathcal{N})}
	:=
	\Z_{(p_1p_2\cdots p_s)},
	\quad
	\Z^{(\mathcal{N})}
	:=
	\Z^{(p_1\cdots p_s)},
	\]
	similarly for $\A_{(\Ncal)}$, etc. For a number field $F$, we write $\A_F$ for the ring of adèles over $F$, $\A_{F,f}$ the ring of finite adèles over $F$. We write $\Acal=\A_\Q$ and $\A_f=\A_{\Q,f}$.

	For a group scheme $G$, we write $Z_G$ for the center of $G$.

	\section{Polarized abelian varieties}\label{Polarized abelian varieties}
	We recall the notion of polarized abelian varieties with level structures. A very convenient reference for this section is \cite[§2]{Yu2010}.

	We fix a non-zero  free $\Z$-module of finite type $L$ and write $V=L\otimes_{\Z}\mathbb{Q}$. We assume that there is a non-degenerate symplectic bilinear pairing
	\[
	\langle-,-\rangle\colon L\times L\rightarrow\Z.
	\]
	We do not assume that $L$ is unimodular with respect to $\langle-,-\rangle$.

	We define the similitude symplectic group scheme $\bfGSp_L$ over $\Z$ associated to the symplectic pairing $(L,\langle-,-\rangle)$: for any $\Z$-algebra $R$,
	\[
	\bfGSp_L(R)
	:=
	\{
	(g,r)\in\mathrm{GL}_{R}(L\otimes R)\times R^\times
	\mid
	\langle gx,gy\rangle=\mu(g)\langle x,y\rangle,
	\forall
	x,y\in L\otimes R
	\}.
	\]
	Note that $\bfGSp_L$ is not necessarily a smooth group scheme over $\mathrm{Spec}(\Z)$.
	Since $L\neq0$, the pair $(g,\mu(g))$ is uniquely determined by the first component $g$, so we will also write $g$ for $(g,\mu(g))$. Moreover $\mu$ is the similitude factor of $g$ and defines a morphism of group schemes over $\mathrm{Spec}(\Z)$:
	\[
	\mu
	\colon
	\bfGSp_L
	\rightarrow
	\mathbb{G}_m.
	\]
	The isometry subgroup $\bfSp_L$ of $\bfGSp_L$ is given by
	\[
	\bfSp_L(R)
	:=
	\{
	g\in\bfGSp_L(R)
	\mid
	\mu(g)=1
	\}.
	\]

	For an abelian scheme $A$ over a scheme $S$, we write $A^\vee$ for the dual abelian scheme of $A$. For two abelian schemes $A,B$ over $S$, we write
	\[
	\mathrm{Hom}_{S}(A,B)^\circ=\mathrm{Hom}_{S}(A,B)\otimes_{\Z}\mathbb{Q}.
	\]
	Similarly $\mathrm{End}_S(A)^\circ=\mathrm{End}_S(A)\otimes_{\Z}\mathbb{Q}$.

	A \textit{polarized abelian scheme over $S$} is a pair $\underline{A}=(A,\lambda)$ where $A$ is an abelian scheme over $S$ and $\lambda\colon A\rightarrow A^\vee$ is a polarization. We assume in the following
	\[
	\mathrm{dim}_{\mathbb{Q}}(V)
	=
	2\mathrm{dim}_S(A).
	\]
	In other words, for $S=\mathrm{Spec}(\Z)$, these $\uA$ are objects in the Shimura varieties of PEL type attached to the PEL datum $(\mathbb{Q},\dagger=\mathrm{Id}_{\mathbb{Q}},V,\langle-,-\rangle)$.

	Let $\uA=(A,\lambda)$ be as above. For any prime number $\ell$, we write $A[\ell^\infty]$ for the $\ell$-divisible group associated to the abelian scheme $A$ and similarly
	\[
	\underline{A}[\ell^\infty]=(A[\ell^\infty],\lambda_\ell),
	\]
	the $\ell$-divisible group associated to $\underline{A}$. Here $\lambda_\ell\in\mathrm{Hom}_{S}(A[\ell^\infty],A^\vee[\ell^\infty])$ is the polarization induced from $\lambda$ (when there is no confusion, we also write $\lambda$ for $\lambda_\ell$). For any two polarized abelian varieties $\uA_1$ and $\uA_2$ over a field $K$, we write
	\begin{enumerate}
		\item
		$\QISG_K(\uA_1,\uA_2)$ (resp. $\Isom_K(\uA_1,\uA_2)$) for the set of quasi-isogenies (resp. isomorphisms) $\phi\colon A_1\rightarrow A_2$ over $K$ such that $\phi^\ast\lambda_2=\mu\lambda_1$ for some $\mu\in\Q^\times$.

		\item 
		For a prime $\ell$, $\QISG_{K,\ell}(\uA_1,\uA_2)$ is the subset of $\QISG_{K}(\uA_1,\uA_2)$ consisting of $\phi$ whose degree is of the form $\ell^r$ for some $r\in\Z$.

		\item 
		$\QISG_K(\uA_1[\ell^\infty],\uA_2[\ell^\infty])$ (resp. $\Isom_K(\uA_1[\ell^\infty],\uA_2[\ell^\infty])$) for the set of quasi-isogenies (resp. isomorphisms) $\phi\colon A_1[\ell^\infty]\rightarrow A_2[\ell^\infty]$ such that $\phi^\ast\lambda_2=\mu\lambda_1$ for some $\mu\in\Q_\ell^\times$, resp. $\mu\in\Z_\ell^\times$.
	\end{enumerate}
	Here $\phi^\ast\lambda_2:=\phi^\vee\lambda_2\phi$.

	We fix in the following a polarized abelian variety $\uA_0=(A_0,\lambda_0)$ over $K$. 
	We denote $\bfG_{\uA_0}$ for the group scheme over $\Z$ of automorphisms of $\uA_0$ defined as follows: for any $\Z$-algebra $R$,
	\[
	\bfG_{\uA_0}(R)
	:=
	\{
	g\in\mathrm{End}_K(\uA_0)\otimes_{\Z}R
	|
	g'g=\mu(g)\in R^\times
	\}.
	\]
	Here $g\mapsto g':=\lambda_0^{-1}\circ g^\vee\circ\lambda_0$ is the Rosati involution induced by $\lambda_0$. 
	We write $\bfG_{\uA_0}^1$ for the subgroup scheme of $\bfG_{\uA_0}$ defined by
	\[
	\bfG_{\uA_0}^1(R)=\{g\in \bfG_{\uA_0}(R)\mid\mu(g)=1\}.
	\]
	It is easy to see
	\[
	\bfG_{\uA_0}(\mathbb{Q})
	\simeq
	\QISG_K(\uA_0,\uA_0),
	\quad
	\bfG_{\uA_0}(\Z_{(\ell)})
	\simeq
	\QISG_{K,\ell}(\uA_0\ldiv,\uA_0\ldiv).
	\]

	From now on we assume $S=\mathrm{Spec}(K)$ with $K$ an algebraically closed field.

	Let $U$ be a compact open subgroup of $\QISG_K(\uA_0\pdiv,\uA_0\pdiv)$.
	A \textit{$U$-level structure} on a polarized abelian scheme $\uA=(A,\lambda)$ over $K$ is a \textit{non-empty} $U$-orbit of isomorphisms in
	$\QISG_K(\uA_0\pdiv,\uA\pdiv)$.	In the following we also write $\uA$ for the triple $\uA=(A,\lambda,\overline{\eta})$ where $(A,\lambda)$ is a polarized abelian variety and $\overline{\eta}$ is a $U$-level structure on $(A,\lambda)$. Moreover, we write
	\[
	\uA_0=(A_0,\lambda_0,\overline{\eta}_0),
	\]
	where $\overline{\eta}_0$ is the $U$-orbit of the identity in $\Isom_K(\uA_0\pdiv,\uA_0\pdiv)$. As in the case of $(A,\lambda)$, for two polarized abelian varieties with $U$-level structures $\uA_1=(A_1,\lambda_1,\overline{\eta}_1),\uA_2=(A_2,\lambda_2,\overline{\eta}_2)$, we write
	\begin{enumerate}
		\item
		$\QISG_K(\uA_1,\uA_2)$ is the same as $\QISG_{K}((A_1,\lambda_1),(A_2,\lambda_2))$.

		\item 
		$\QISG_{K,\ell}(\uA_1,\uA_2)$ is the same as $\QISG_{K,\ell}((A_1,\lambda_1),(A_2,\lambda_2))$.

		\item 
		For a prime $\ell\neq p$, $\Isom_K(\uA_1\ldiv,\uA_2\ldiv)$ is the same as $\Isom_K((A_1,\lambda_1)\ldiv,(A_2,\lambda_2)\ldiv)$).

		\item 
		$\Isom_K(\uA_1\pdiv,\uA_2\pdiv)$ (resp. $\Isom_K(\uA_1,\uA_2)$) is the set of $\phi$ in $\Isom_K((A_1,\lambda_1)\pdiv,(A_2,\lambda_2)\pdiv)$ (resp. $\Isom_K((A_1,\lambda_1),(A_2,\lambda_2))$) such that $\phi_\ast\overline{\eta}_1=\mu\overline{\eta}_2$ for some $\mu\in\Q_p^\times$ (resp. $\mu\in\Q^\times$).
	\end{enumerate}
	We write again $\bfG_{\uA_0}=\bfG_{(A_0,\lambda_0)}$. Then we have ($\ell\neq p$)
	\[
	\bfG_{\uA_0}(\mathbb{Q})
	\simeq
	\QISG_{K}(\uA_0,\uA_0),
	\quad
	\bfG_{\uA_0}(\Z[1/\ell])
	\simeq
	\QISG_{K,\ell}(\uA_0\ldiv,\uA_0\ldiv).
	\]

	\section{$p$-adic Hecke orbits of $\uA_0$}\label{p-adic Hecke orbit of A}
	We fix a prime number $p$ as in the preceding section.
	Let $\uA_0=(A_0,\lambda_0,\overline{\eta}_0)$ be a polarized abelian variety over an algebraically closed field $K=\overline{K}$ with $U$-level structure ($U$ is a compact open subgroup of $\QISG_{K}(\underline{A}_0\pdiv,\underline{A}_0\pdiv)$).
	\begin{definition}\label{p-adic Hecke orbit}
		The \emph{$p$-adic Hecke orbit of $\uA_0$}, denoted by
		\[
		\mathcal{H}_p(\uA_0)
		\]
		is the set of isomorphism classes of polarized abelian varieties $\uA$ with a $U$-level structure over $K$ such that
		\[
		\QISG_{K,p}(\uA_0,\uA)\neq\emptyset.
		\]
	\end{definition}

	Then it is not hard to see
	\begin{theorem}\label{description of H_p(A_0)}
		We have a natural bijection
		\[
		\psi
		\colon
		\mathcal{H}_p(\uA_0)
		\simeq
		\QISG_{K,p}(\uA_0,\uA_0)
		\backslash
		\QISG_K(\uA_0[p^\infty],\uA_0[p^\infty])
		/U.
		\]
	\end{theorem}
	\begin{proof}
		The proof follows closely that of \cite[Theorem 2.2]{Yu2010}.
		We define the map $\psi$
		as follows: for any $\uA\in\mathcal{H}_p(\uA_0)$, we fix $\varphi\in\QISG_{K,p}(\uA,\uA_0)$ and $\alpha_p\in\QISG_K(\uA_0[p^\infty],\uA[p^\infty])$, then the composition
		\[
		\varphi\circ\alpha_p\in\QISG_K(\uA_0[p^\infty],\uA_0[p^\infty]).
		\]
		We define $\psi(\uA)$ to be the image of $\varphi\circ\alpha_p$ in the quotient
		\[
		\psi(\uA)
		\in
		\QISG_{K,p}(\uA_0,\uA_0)
		\backslash
		\QISG_K(\uA_0[p^\infty],\uA_0[p^\infty])
		/U.
		\]
		Moreover, changing $\alpha_p$ amounts to multiplying $\varphi\circ\alpha_p$ on the right by an element in $U$; changing $\varphi$ amounts to multiplying $\varphi\circ\alpha_p$ on the left by an element in $\QISG_{K,p}(\uA_0,\uA_0)$, thus the map $\psi$ is well-defined and is clearly injective.

		Next we show that $\psi$ is surjective: for any $\phi_p\in\QISG_{K,p}(\uA_0\pdiv,\uA_0\pdiv)$, let $N=p^r$ with $r\in\mathbb{N}$ such that
		\[
		f_p:=N\phi_p^{-1}\colon\uA_0[p^\infty]\rightarrow\uA_0[p^\infty]
		\]
		is an isogeny. Then we write $H=\mathrm{Ker}(f_p)$, $A=A_0/H$ and $\pi\colon A_0\rightarrow A$ for the natural projection map. We define
		\[
		\lambda:=\mu(\phi_p)^{-1}(N\pi^{-1})^\ast\lambda_0,
		\]
		a quasi-polarization of $A$. Since $f_p$ and $\pi_p\colon A_0[p^\infty]\rightarrow A[p^\infty]$ have the same kernel, there is an isomorphism $\alpha_p\colon A_0[p^\infty]\rightarrow A[p^\infty]$ such that $\alpha_p\circ f_p=\pi_p$. In particular, 
		\begin{align*}
			\alpha_p^\ast\lambda
			&
			=
			\alpha_p^\ast(N\pi_p^{-1})^\ast\lambda_0
			=
			(N\pi_p^{-1}\alpha_p)^\ast\lambda_0
			\\
			&
			=
			(Nf_p^{-1})^\ast\lambda_0
			=
			\phi_p^{-1}\lambda_0=\lambda_0
		\end{align*}
		and thus $\lambda=(\alpha_p^{-1})^\ast\lambda$. This shows that $\lambda$ is a \emph{polarization} of $A$. Now we set
		\[
		\varphi
		:=
		N\pi^{-1}
		\colon
		A\rightarrow A_0.
		\]
		Thus $\varphi\circ\alpha_p=\phi_p$, a quasi-isogeny whose degree lies in $p^{\Z}$. This gives $\uA\in\mathcal{H}_p(\uA_0)$, proving the surjectivity of $\psi$.
	\end{proof}

	\subsection{The case $K=\overline{\mathbb{Q}}$}\label{case K=Qbar}
	Recall we have a symplectic pairing $\langle-,-\rangle$ on $L$.
	We assume that there is an isomorphism of $\widehat{\Z}$-modules
	\[
	L\otimes_{\Z}\widehat{\Z}
	\simeq
	\widehat{T}(A_0):=\prod_{\ell}T_\ell(A_0)
	\]
	such that we have a commutative diagram of $\widehat{\Z}$-modules
	\[
	\begin{tikzcd}
		(L\otimes_{\Z}\widehat{\Z})
		\times
		(L\otimes_{\Z}\widehat{\Z})
		\arrow[r,"{\langle-,-\rangle}"]
		\arrow[d,"\simeq"]
		&
		\widehat{\Z}[1]
		\arrow[d,"\simeq"]
		\\
		\widehat{T}(A_0)
		\times
		\widehat{T}(A_0)
		\arrow[r,"{\langle-,-\rangle_{\lambda_0}}"]
		&
		\widehat{T}(\mathbb{G}_m)
	\end{tikzcd}
	\]
	Here $T_\ell(A_0)$ is the $\ell$-adic Tate module of $A_0$ and $\langle-,-\rangle_{\lambda_0}$ is the Weil pairing on the Tate modules of $A_0$ induced by the polarization $\lambda_0$ and $\widehat{\Z}[1]\simeq\widehat{T}(\mathbb{G}_m)$ is an isomorphism of $\widehat{\Z}$-modules (\cite[definition 1.3.6.1]{Lan2013}). As a result, for any prime $\ell$, we have the following isomorphisms:
	\begin{align*}
		\bfGSp_L(\Z_\ell)
		&
		\simeq
		\Isom_K((A_0,\lambda_0)\ldiv,(A_0,\lambda_0)\ldiv)
		\supset
		\Isom_K(\uA_0\ldiv,\uA_0\ldiv),
		\\
		\bfGSp_L(\mathbb{Q}_\ell)		
		&
		\simeq
		\QISG_{K}((A_0,\lambda_0)\ldiv,(A_0,\lambda_0)\ldiv)
		=
		\QISG_{K}(\uA_0\ldiv,\uA_0\ldiv).
	\end{align*}
	So in the following we will identify the restricted product $\prod_{\ell}'\QISG_{K}(\uA_0\ldiv,\uA_0\ldiv)$ with $\bfGSp_L(\A_f)$.
	In particular, we will view $U$ as a subgroup of $\bfGSp_L(\widehat{\Z})$.

	From Theorem \ref{description of H_p(A_0)}, we get the following bijection
	\begin{equation}\label{p-adic Hecke orbit of A, char=0}
		\Theta
		\colon
		\mathcal{H}_p(\uA_0)
		\simeq
		\bfG_{\uA_0}(\Z[1/p])
		\backslash
		\bfGSp_L(\mathbb{Q}_p)
		/U.
	\end{equation}
	Note that $Z_{\bfG_{\uA_0}}(\mathbb{Q}_p)$ acts naturally on the double quotient $\Gbf_{\underline{A}_0}(\Z[1/p])\backslash\bfSp_L(\mathbb{Q}_p)/U$ by left multiplication. We equip $\mathcal{H}_p(\uA_0)$ with an action of $Z_{\bfG_{\uA_0}}(\mathbb{Q}_p)$ by \emph{transport de structure} via the above bijection.

	\begin{example}
		Suppose $A_0$ is an abelian variety with CM by an order $\mathcal{O}$ of a CM field $K$, that is, we have an embedding $\mathcal{O}\hookrightarrow\mathrm{End}_K(A_0)$ and $\mathrm{dim}_{\mathbb{Q}}(K)=2\mathrm{dim}(A_0)$. Then for any $\Z$-algebra $R$,
		\[
		\bfG_{\uA_0}^1(R)
		=
		\{
		g\in\mathcal{O}\otimes_{\Z}R
		\mid
		g'g=1
		\}.
		\]
		In particular, $\bfG_{\uA_0}^1$ is a commutative group scheme over $\Z$.
	\end{example}

	\subsection{The case $K=\overline{\mathbb{F}}_\ell$}\label{case K=F_l}
	Assume $p\neq\ell$.
	We let $\uA_0$ be as in the preceding subsection ($K=\overline{\mathbb{Q}}$).
	We fix a place $v=v(\ell)$ of $\overline{\mathbb{Q}}$ over $\ell$ and assume that $\uA_0$ has \emph{supersingular} good reduction at $v$. We write
	\[
	\uA_{0}^{(\ell)}=(A_{0}^{(\ell)},\lambda_0,\overline{\eta})
	\]
	for the reduction modulo $v$ of $\uA_0$. Note that here we still use $\lambda_0$ to denote the polarization of $A_{0}^{(\ell)}$ induced by the polarization $\lambda_0$ on $A_0$ and similarly for the $U$-level structure. Then we again have isomorphisms
	\[
	\bfGSp_L(\widehat{\Z}^{(\ell)})
	\simeq
	\prod_{q\neq\ell}
	\Isom_K(\uA_{0}^{(\ell)}[q^\infty],\uA_{0}^{(\ell)}[q^\infty]).
	\]
	Similarly for the $\A_f^{(\ell)}$-points of $\bfGSp_L$. Since $\uA_{0}^{(\ell)}$ is supersingular, by a theorem of Deligne, Ogus and Shioda (\cite{Shioda1978}), $\uA_{0}^{(\ell)}$ is \emph{isogenous} to $E(\ell)^n$ where $E(\ell)$ is a supersingular elliptic curve over $K$. We write
	\[
	B(\ell)=\mathrm{End}_K(E(\ell))^\circ,
	\]
	the unique quaternion algebra over $\mathbb{Q}$ ramified exactly at $\ell$ and $\infty$. It follows
	\[
	\mathrm{Mat}_n(B(\ell))
	\simeq
	\mathrm{End}_K(A_{0}^{(\ell)})^\circ.
	\]
	The polarization $\lambda_0$ on $A_{0}^{(\ell)}$ induces the (Rosati) involution on $\mathrm{Mat}_n(B(\ell))$, which is positive definite. Therefore we have
	\[
	\bfG_{\uA_{0}^{(\ell)}}^1(\mathbb{Q})
	=
	\{
	g\in\mathrm{Mat}_n(B(\ell))
	\mid
	g'g=1_n
	\},
	\]
	which is a quaternionic special unitary group compact at infinity (that is, $\bfG_{\uA_{0}^{(\ell)}}^1(\mathbb{R})$ is compact). Moreover, for any prime $q\neq\ell$, we have an isomorphism
	\[
	\bfG_{\uA_{0}^{(\ell)}}(\Z_q)
	\simeq
	\Isom_K(\uA_{0}^{(\ell)}[q^\infty],\uA_{0}^{(\ell)}[q^\infty]).
	\]
	\begin{remark}
		In the following we will identify the groups $\bfGSp_L(\Z_p)\simeq\bfG_{\uA_{0}^{(\ell)}}(\Z_p)$ (similarly for $\mathbb{Q}_p$-points instead of $\Z_p$-points). In particular, we view subgroups of $\bfGSp_L(\mathbb{Q}_p)$ also as subgroups of $\bfG_{\uA_{0}^{(\ell)}}(\mathbb{Q}_p)$ (and vice versa) without further comments.
	\end{remark}

	From Theorem \ref{description of H_p(A_0)}, we get the following bijection
	\begin{equation}\label{p-adic Hecke orbit of A, char=ell}
		\Theta_\ell
		\colon
		\mathcal{H}_p(\uA_{0}^{(\ell)})
		\simeq
		\Gbf_{\uA_{0}^{(\ell)}}(\Z[1/p])
		\backslash
		\bfGSp_L(\mathbb{Q}_p)
		/U.
	\end{equation}

	\subsection{Reduction of $p$-adic Hecke orbit of $\uA_0$}\label{Reduction of p-adic Hecke orbit of A}
	We fix a non-empty finite set $\mathcal{S}$ of primes $\ell$ different from $p$ and for each $\ell\in\Scal$, we fix a place $v(\ell)$ of $\overline{\Q}$ over $\ell$. Write $k_v$ for the residue field of $\overline{\mathbb{Q}}$ at $v=v(\ell)$ (so $k_v\simeq\overline{\F}_\ell$). We assume that $\uA_0$ has \emph{supersingular} good reduction at $v$. For each $\ell\in\mathcal{S}$, fix a non-empty finite subset $\mathcal{T}_\ell$ of $Z_{\bfG_{\uA_0}}(\A_f)$ and write
	\[
	\mathcal{T}=\{\mathcal{T}_\ell\mid\ell\in\mathcal{S}\}.
	\]
	Then by \cite[Corollary 2, p.493]{SerreTate1968}, we know that all elements $\uA$ in $\mathcal{H}_p(\uA_0)$ have good reduction at $v$. We write $\uA^{(\ell)}$ for the reduction modulo $v$ of $\uA$.
	Therefore we have a well-defined map induced by the reduction modulo $v$:
	\[
	\bfR_{\ell}
	\colon
	\mathcal{H}_p(\uA_0)
	\mapsto
	\mathcal{H}_p(\uA_{0}^{(\ell)}),
	\quad
	\uA
	\mapsto
	\uA^{(\ell)}.
	\]
	Then we define the following map of simultaneous reduction of the $p$-adic Hecke orbit of $\uA_0$
	\begin{equation*}
		\bfR_{\mathcal{S},\mathcal{T}}
		\colon
		\mathcal{H}_p(\uA_0)
		\rightarrow
		\prod_{\ell\in\mathcal{S}}
		\prod_{\sigma\in\mathcal{T}_\ell}\mathcal{H}_p(\uA_{0}^{(\ell)}),
		\quad
		\uA
		\mapsto
		((\sigma\uA)_{\ell})_{\ell\in\mathcal{S},\sigma\in\mathcal{T}_\ell}.
	\end{equation*}

	For each $\ell\in\mathcal{S}$, we have a partition of $\mathcal{T}_\ell$, given by the pre-images of $\overline{\mathcal{T}}(\ell)$ under the natural projection map
	$\pi\colon
	Z_{\bfG_{\uA_0}}(\mathbb{Q}_p)
	\rightarrow
	Z_{\bfG_{\uA_0}}(\mathbb{Q}_p)/Z_{\bfG_{\uA_0}}(\mathbb{Q})$:
	\begin{equation}\label{partition of T_l}
		\mathcal{T}_\ell
		=
		\bigsqcup_{i=1}^{r(\ell)}
		\mathcal{T}_{\ell,i}
		\quad
		\text{where }
		\sigma,\sigma'\in\mathcal{T}_{\ell,i}
		\Leftrightarrow
		\pi(\sigma_p)
		=
		\pi(\sigma'_p).
	\end{equation}
	In other words, $\sigma$ and $\sigma'$ lie in the same $\mathcal{T}_{\ell,i}$ if and only if $\sigma_p^{-1}\sigma_p'\in Z_{\bfG_{\uA_0}}(\mathbb{Q})$.

	We write
	\begin{align*}
		\Gamma_p(\ell)
		:=
		\bfG_{\uA_{0}^{(\ell)}}(\Z[1/p]),
		\quad
		\Gamma_p(\ell)_\sigma
		:=
		\sigma^{-1}
		\Gamma_p(\ell)
		\sigma,
		\\
		\Gamma_p^1(\ell)
		:=
		\bfG_{\uA_{0}^{(\ell)}}^1(\Z[1/p]),
		\quad
		\Gamma_p^1(\ell)_\sigma
		:=
		\sigma^{-1}
		\Gamma_p^1(\ell)
		\sigma.
	\end{align*}
	Then we have the following bijection:
	\[    
	\Theta_{\ell,\sigma}
	\colon
	\Gamma_p(\ell)_{\sigma}
	\backslash
	\bfGSp_L(\mathbb{Q}_p)
	/U
	\rightarrow
	\mathcal{H}_p(\uA_{0}^{(\ell)})
	\quad
	\Gamma_p(\ell)_{\sigma}gU
	\mapsto
	\Theta_{\ell}^{-1}(\sigma g).
	\]
	\begin{lemma}\label{bijections between PSp_L and PGSp_L}
		Suppose that $\Z_p^\times\subset U$, then we have the following natural bijections
		\[
		\Gamma_p^1(\ell)_\sigma\{\pm1\}\backslash\bfSp_L(\Q_p)/U^1
		\simeq
		\Gamma_p^1(\ell)_\sigma\backslash\bfSp_L(\Q_p)/U^1
		\simeq
		\Gamma_p(\ell)_\sigma\backslash\bfGSp_L(\Q_p)/U
		\simeq
		\Gamma_p(\ell)_\sigma\Q_p^\times\backslash\bfGSp_L(\Q_p)/U.
		\]
		Here $U^1=U\cap\bfSp_L(\Q_p)$.
	\end{lemma}
	\begin{proof}
		Clearly we have $\pm p^\Z\subset\Gamma_p(\ell)_\sigma$ and $\{\pm1\}\subset\Gamma_p^1(\ell)_\sigma$. Moreover, by assumption, $\Z_p^\times\subset U$ and $\{\pm1\}\subset U^1$, thus we have the first bijection and the last bijection.

		Note that the reduced norm of $B(\ell)$ is the subset of $\Q$, consisting of $x\in\Q$ which is a square in $\R$ and $\Q_\ell$.
		Then by \cite[p.90, Corollaire 5.9]{Vigneras1980}, we have $\det(\Gamma_p(\ell)_\sigma)=\det(\Gamma_p(\ell))\supset p^\Z$, which gives the middle bijection.
	\end{proof}

	For any non-empty subset $\mathcal{T}'$ of $\mathcal{T}_{\ell}$,
	we define the twisted diagonal map
	\begin{equation}\label{twisted diagonal map}
		\widetilde{\Delta}^{\mathcal{T}'}
		\colon
		\mathcal{H}_p(\uA_{0}^{(\ell)})
		\rightarrow
		\prod_{\sigma\in\mathcal{T}'}
		\mathcal{H}_p(\uA_{0}^{(\ell)})
	\end{equation}
	to be the following composition map
	\[
	\begin{tikzcd}
		\mathcal{H}_p(\uA_{0}^{(\ell)})
		\arrow[rrr,"\Theta_{\ell,1}^{-1}"]
		\arrow[d,dashrightarrow,"\widetilde{\Delta}^{\mathcal{T}'}"]
		&&&
		\Gamma_p(\ell)
		\backslash
		\bfGSp_L(\mathbb{Q}_p)
		/
		U
		\arrow[d,"g\mapsto(\sigma g)_{\sigma\in\mathcal{T}'}"]
		\\
		\prod_{\sigma\in\mathcal{T}'}
		\mathcal{H}_p(\uA_{0}^{(\ell)})
		&&&
		\prod_{\sigma\in\mathcal{T}'}
		\Gamma_p(\ell)_{\sigma}
		\backslash
		\bfGSp_L(\mathbb{Q}_p)
		/
		U
		\arrow[lll,"\prod_{\sigma\in\mathcal{T}'}\Theta_{\ell,\sigma}"']
	\end{tikzcd}
	\]

	Here is one of the main theorems of this article
	(Theorem \ref{main result})
	\begin{theorem}\label{R_{S,T} is an isomorphism}
		The image of the simultaneous reduction map
		\[
		\bfR_{\mathcal{S},\mathcal{T}}
		\colon
		\mathcal{H}_p(\uA_0)
		\rightarrow
		\prod_{\ell\in\mathcal{S},\sigma\in\mathcal{T}_{\ell}}
		\mathcal{H}_p(\uA_{0}^{(\ell)}),
		\quad
		\uA
		\mapsto
		((\sigma\uA)_{\ell})_{\ell\in\mathcal{S},\sigma\in\mathcal{T}_\ell}
		\]
		is given by
		$\prod_{\ell\in\mathcal{S}}
		\prod_{i=1}^{r(\ell)}
		\widetilde{\Delta}^{\mathcal{T}_{\ell,i}}
		\left(
		\mathcal{H}_p(\uA_{0}^{(\ell)})
		\right)$. In particular, it is in bijection with
		$\prod_{\ell\in\mathcal{S},\overline{\sigma}\in\overline{\mathcal{T}}_\ell}
		\mathcal{H}_p(\uA_{0}^{(\ell)})$
		as stated in Theorem \ref{main result}.
	\end{theorem}

	Let $\mathcal{S}$ and $\mathcal{T}$ be as in the above theorem, then using the descriptions of the sets $\mathcal{H}_p(\uA_0)$ and $\mathcal{H}_p(\uA_{0}^{(\ell)})$ as in (\ref{p-adic Hecke orbit of A, char=0}) and (\ref{p-adic Hecke orbit of A, char=ell}), as well as Lemma \ref{bijections between PSp_L and PGSp_L}, we have the following commutative diagram
	\begin{equation*}
		\begin{tikzcd}
			\mathcal{H}_p(\uA_0)
			\arrow[r,"\bfR_{\ell}"]
			\arrow[d,"\Theta"]
			&
			\mathcal{H}_p(\uA_{0}^{(\ell)})
			\arrow[d,"\Theta_\ell"]
			\\
			\Gbf_{\underline{A}_0}(\Z[1/p])\backslash\bfGSp_L(\mathbb{Q}_p)/U
			\arrow[r]
			&
			\Gamma_p(\ell)\backslash
			\bfGSp_L(\mathbb{Q}_p)/U.
		\end{tikzcd}
	\end{equation*}
	Here the bottom horizontal map is the natural projection induced by the inclusion $\bfG_{\uA_0}\subset\bfG_{\uA_{0}^{(\ell)}}$.
	Similarly, we have
	\begin{equation*}
		\begin{tikzcd}
			\mathcal{H}_p(\uA_0)
			\arrow[r,"\bfR_{\Scal,\Tcal}"]
			\arrow[d,"\Theta"]
			&
			\prod_{\ell\in\Scal,\sigma\in\Tcal}\mathcal{H}_p(\uA_{0}^{(\ell)})
			\arrow[d,"\Theta_\ell"]
			\\
			\Gbf_{\underline{A}_0}(\Z[1/p])\backslash\bfGSp_L(\mathbb{Q}_p)/U
			\arrow[r]
			&
			\prod_{\ell\in\Scal,\sigma\in\Tcal}
			\Gamma_p(\ell)_\sigma\backslash
			\bfGSp_L(\mathbb{Q}_p)/U.
		\end{tikzcd}
	\end{equation*}
	We deduce from this commutative diagram that Theorem \ref{R_{S,T} is an isomorphism} follows from the theorem below

	\begin{theorem}\label{surjectivity of Pi_R}
		The closure of the image of the simultaneous \emph{projection} map
		\[
		\bfSp_L(\mathbb{Q}_p)
		\rightarrow
		\prod_{\ell\in\mathcal{S},\sigma\in\mathcal{T}_{\ell}}
		\Gamma_p^1(\ell)_{\sigma}
		\backslash
		\bfSp_L(\mathbb{Q}_p)
		\]
		is equal to
		\[
		\prod_{\ell\in\mathcal{S}}
		\prod_{i=1}^{r(\ell)}
		\Gamma_p^1(\ell)_{i}
		\backslash
		\left(
		\Gamma_p^1(\ell)_{i}
		\Delta^{\mathcal{T}_{\ell,i}}(\bfSp_L(\mathbb{Q}_p))
		\right),
		\]		
		where
		$\Gamma_p^1(\ell)_{i}
		=
		\prod_{\sigma\in\mathcal{T}_{\ell,i}}
		\Gamma_p^1(\ell)_{\sigma}$ and $\Delta^{\Tcal_{\ell,i}}\colon\bfSp_L(\Q_p)\to\bfSp_L(\Q_p)^{\Tcal_{\ell,i}}$ is the diagonal embedding.
	\end{theorem}
	The proof of this theorem will be given in §\ref{Proof of the main result}.

	\section{Prime-to-$N$ Hecke orbits of $\uA_0$}\label{prime-to-N Hecke orbits of A_0}
	In this section we are concerned with prime-to-$N$ Hecke orbits of $\uA_0$. This is almost the same as the preceding section, up to changing `$p$-adic' to `prime-to-$N$' everywhere. So we will be very brief in this section and state only the main results.

	We fix a polarized abelian variety $\uA_0=(A_0,\lambda_0)$ over an algebraically closed field $K$ and a positive integer $N$. We write $n=\mathrm{dim}(A_0)$. Let $U$ be a compact open subgroup of the restricted product
	$\prod_{\ell\nmid N}'\QISG_K(\uA_0\ldiv,\uA_0\ldiv)$. Then we write
	\[
	\uA_0=(A_0,\lambda_0,\overline{\eta}_0)
	\]
	where $\overline{\eta}_0$ is the $U$-orbit of the identity in $\prod_{\ell\nmid N}\Isom_K(\uA_0\ldiv,\uA_0\ldiv)$. We define the two sets $\QISG_{K}(\uA_1,\uA_2)$ and $\QISG_{K}(\uA_1\pdiv,\uA_2\pdiv)$ as in the preceding section. Moreover, we write
	\[
	\QISG_{K}^{(N)}(\uA_1,\uA_2)
	\]
	for the subset of $\QISG_{K}(\uA_1,\uA_2)$ consisting of those $\phi$ whose degree does not contain prime factors dividing $N$.

	The \emph{prime-to-$N$ Hecke orbit of $\uA_0$}, denoted by
	\[
	\mathcal{H}^{(N)}(\uA_0),
	\]
	is the set of 
	isomorphism classes of polarized abelian varieties $\uA$ with a $U$-level structure over $K$ such that
	\[
	\QISG_{K}^{(N)}(\uA_0,\uA)\neq\emptyset.
	\]
	Then we have a natural bijection
	\[
	\psi
	\colon
	\mathcal{H}^{(N)}(\uA_0)
	\simeq
	\QISG_{K}^{(N)}(\uA_0,\uA_0)
	\big\backslash
	\prod_{\ell\nmid N}\QISG_{K}(\uA_0\ldiv,\uA_0\ldiv)
	\big/
	U.
	\]

	We write
	\begin{align*}
		\Gamma^{(N)}(\ell)
		=
		\bfG_{\uA_{0}^{(\ell)}}(\Z^{(N)}),
		\quad
		\Gamma^{(N)}(\ell)_\sigma
		:=
		\sigma^{-1}\Gamma^{(N)}(\ell)\sigma,
		\\
		\Gamma^{(N),1}(\ell)
		=
		\bfG_{\uA_{0}^{(\ell)}}^1(\Z^{(N)}),
		\quad
		\Gamma^{(N),1}(\ell)_\sigma
		:=
		\sigma^{-1}\Gamma^{(N)}(\ell)\sigma.
	\end{align*}
	For $K=\overline{\mathbb{Q}}$, assume that $\uA_0$ satisfies the same conditions  as in §\ref{case K=Qbar}, then the above bijection becomes
	\[
	\Theta
	\colon
	\mathcal{H}^{(N)}(\uA_0)
	\simeq
	\Gamma^{(N)}(\ell)
	\backslash
	\bfGSp_L(\A_f^{(N)})/U.
	\]

	For $\ell\nmid N$, then we have a bijection
	\[
	\Theta_\ell
	\colon
	\mathcal{H}^{(N)}(\uA_{0}^{(\ell)})
	\simeq
	\bfG_{\uA_{0}^{(\ell)}}(\Z^{(N)})
	\backslash
	\bfGSp_L(\A_f^{(N)})/U.
	\]

	Now we fix a non-empty finite set $\mathcal{S}$ of primes $\ell\nmid N$ and for each $\ell\in\mathcal{S}$, fix a non-empty finite set $\mathcal{T}_\ell$ of elements in $Z_{\bfG_{\uA_0}}(\A_f)$ and write $\mathcal{T}=\{\mathcal{T}_\ell|\ell\in\mathcal{S}\}$. We have a partition of each $\mathcal{T}_\ell$
	\[
	\mathcal{T}_\ell
	=
	\bigsqcup_{i=1}^{r(\ell)}\mathcal{T}_{\ell,i}
	\quad
	\text{ where }
	\sigma,\sigma'\in\mathcal{T}_{\ell,i}
	\Leftrightarrow
	\sigma^{-1}\sigma'
	\in
	Z_{\bfG_{\uA_0}}(\mathbb{Q})\prod_{q\mid N}Z_{\bfG_{\uA_0}}(\mathbb{Q}_q).
	\]
	We define the twisted diagonal map $\widetilde{\Delta}^{\mathcal{T}'}$ in a similar way to the preceding section.

	Then the main result of this section is
	\begin{theorem}\label{R_{S,T}^N is an isomorphism}
		The image of the simultaneous reduction map
		\[
		\bfR_{\mathcal{S},\mathcal{T}}^{(N)}
		\colon
		\mathcal{H}^{(N)}(\uA_0)
		\to
		\prod_{\ell\in\mathcal{S},\sigma\in\mathcal{T}_{\ell}}
		\mathcal{H}^{(N)}(\uA_{0}^{(\ell)}),
		\quad
		\uA
		\mapsto
		((\sigma\uA)_{\ell})_{\ell\in\mathcal{S},\sigma\in\mathcal{T}_\ell}
		\]
		is given by
		$\prod_{\ell\in\mathcal{S}}\prod_{i=1}^{r(\ell)}
		\widetilde{\Delta}^{\mathcal{T}_{\ell,i}}(\mathcal{H}^{(N)}(\uA_{0}^{(\ell)}))$, which is in bijection with $\prod_{\ell\in\mathcal{S},\overline{\sigma}\in\overline{\mathcal{T}}(\ell)}
		\mathcal{H}^{(N)}(\uA_{0}^{(\ell)})$ as stated in Theorem \ref{main theorem for prime-to-N Hecke orbit}.
	\end{theorem}

	This theorem follows from
	\begin{theorem}\label{surjectivity of Pi_R^N}
		The closure of the image of the following simultaneous \emph{projection} map
		\[
		\bfSp_L(\A_f^{(N)})
		\rightarrow
		\prod_{\ell\in\mathcal{S},\sigma\in\mathcal{T}_{\ell}}		
		\Gamma^{(N),1}(\ell)_\sigma
		\backslash
		\bfSp_L(\A_f^{(N)})
		\]
		is equal to
		\[
		\prod_{\ell\in\mathcal{S}}
		\prod_{i=1}^{r(\ell)}
		\Gamma^{(N),1}(\ell)_i
		\backslash
		\left(
		\Gamma^{(N),1}(\ell)_i
		\Delta^{\mathcal{T}_{\ell,i}}(\bfSp_L(\A_f^{(N)}))
		\right)
		\]
		where
		$\Gamma^{(N),1}(\ell)_i
		=\prod_{\sigma\in\mathcal{T}_{\ell,i}}\Gamma^{(N),1}(\ell)_{\sigma}$.
	\end{theorem}
	The proof of this theorem will be given in §\ref{Proof of the main result}.

	\section{Commensurability criterion}\label{Commensurability criterion}
	We say that two subgroups $G_1,G_2$ of a group $G$ are \textit{commensurable} if $G_1\bigcap G_2$
	has finite index in	both $G_1$ and $G_2$. The \textit{commensurator} of $G_1$ inside $G$ is the set of elements $g$ in $G$ such that $G_1$ and $gG_1g^{-1}$ are commensurable and we denote it by $\Ccal_G(G_1)$. Suppose $G_2=gG_1g^{-1}$ for some $g\in G$. Then $G_1$ and $G_2$ are commensurable if and only if $g\in\Ccal_G(G_1)$.

	\begin{theorem}\label{commensurable}
		Fix a prime $\ell\in\mathcal{S}$. For any $\sigma_1,\sigma_2\in\mathcal{T}_\ell$, the $p$-th component $(\sigma_1\sigma_2^{-1})_p\notin \Q_p^\times Z_{\bfG_{\uA_0}(\Q)}$ if and only if the subgroups $\Gamma_p(\ell)_{\sigma_1}$ and $\Gamma_p(\ell)_{\sigma_2}$ of $\bfGSp_L(\mathbb{Q}_p)$ are not commensurable.
	\end{theorem}
	\begin{proof}
		For ease of notations, we write
		\begin{align*}
			&
			B=B(\ell),
			\quad
			\sigma=\sigma_1\sigma_2^{-1},
			\\
			&
			\bfG=\bfG_{\uA_{0}^{(\ell)}},
			\quad
			\bfG^1=\bfG_{\uA_{0}^{(\ell)}}^1,
			\\
			&
			\Gamma=\Gamma_p(\ell),
			\quad
			\Gamma_\sigma=\Gamma_p(\ell)_{\sigma},
		\end{align*}
		and
		$\psi_\sigma\colon\bfG(\mathbb{Q}_p)\rightarrow\bfG(\mathbb{Q}_p)$ the inner automorphism sending $h$ to $\sigma_ph\sigma_p^{-1}$. In particular, $\psi_{\sigma}$ sends the subgroup $\Gamma\bigcap\Gamma_\sigma$ of $\bfG(\mathbb{Q})$ to the subgroup $\psi_\sigma(\Gamma\bigcap\Gamma_\sigma)$ of $\bfG(\mathbb{Q})$.

		We first prove the `only if' part. So we assume $\sigma_p
		\notin \Q_p^\times Z_{\bfG_{\uA_0}(\Q)}$ and it suffices to prove that $\Gamma$ and $\Gamma_{\sigma}$ are not commensurable. We prove this by contradiction. Since $\Gamma$ and $\Gamma_\sigma$ are both Zariski dense in $\bfG(\mathbb{Q}_p)$ and $\Gamma,\Gamma_\sigma$ are commensurable, $\Gamma\bigcap\Gamma_\sigma$ is also Zariski dense in $\Gbf(\Q_p)$. Thus we have an automorphism
		\[
		\psi_\sigma
		\colon
		\mathrm{Lie}(\bfG)\rightarrow\mathrm{Lie}(\bfG),
		\]
		where $\mathrm{Lie}(\bfG)$ is the Lie algebra of $\bfG$ over $\mathbb{Q}$ (\emph{cf.} \cite[Exercise 5.2.4]{Morris2001}). Moreover, by the correspondence between algebraic groups and Lie algebras (\cite[§14.1]{Humphreys1975}), we have
		\[
		\mathrm{Ad}(\bfG)=\mathrm{Aut}(\mathrm{Lie}(\bfG)).
		\]
		Here we use the fact that $\mathrm{Aut}(\mathrm{Lie}(\bfG))$ is connected. Since $\bfG$ has no outer automorphisms, we have $\mathrm{Ad}(\bfG)=\mathrm{Aut}(\bfG)$ and thus $\psi_\sigma\in\mathrm{Aut}(\bfG)$.

		We claim that the element $\sigma_p$	also lies in
		$Z_{\bfG(\Q_p)}\bfG(\mathbb{Q})=\mathbb{Q}_p^\times\bfG(\mathbb{Q})$. Assuming this claim, we have
		\[
		\sigma
		\in
		Z_{\bfG_{\uA_0}(\Q_p)}
		\bigcap
		\mathbb{Q}_p^\times\bfG(\mathbb{Q})
		\subset
		\mathbb{Q}_p^\times
		(Z_{\bfG_{\uA_0}(\Q_p)}\bigcap\bfG(\mathbb{Q}))
		=
		\mathbb{Q}_p^\times Z_{\bfG_{\uA_0}(\Q)},
		\]
		which contradicts our assumption $\sigma\notin\Q_p^\times Z_{\bfG_{\uA_0}(\Q)}$ and thus $\Gamma$ and $\Gamma_\sigma$ are not commensurable. Now we prove the claim. Note that $\psi_\sigma$ induces an automorphism of the adjoint quotient $\Pbf\Gbf$ of $\Gbf$, which is a Chevalley group of normal type (in the sense of \cite[§2]{Steinberg1960}, but for infinite Chevalley groups instead of finite ones). By \cite{Humphreys1969}, the automorphism $\psi_\sigma$ of $\Pbf\Gbf(\Q)$ is a composition of inner automorphism $i$, diagonal automorphism $d$, field automorphism $f$ and graph automorphism $g$ of $\Pbf\Gbf(\Q)$ (we refer to \emph{loc.cit} for explanations of these terminologies):
		\[
		\psi_\sigma=gfdi.
		\]
		Moreover, $g$ and $f$ are uniquely determined by $\psi_\sigma$ (this is proved in \cite{Steinberg1960} for the case of finite Chevalley groups. However, it generalizes without modification to infinite Chevalley groups as in \cite{Humphreys1969}) and $d$ is trivial because $\Q$ has no non-trivial automorphism of order $2$.
		On the other hand, $\psi_\sigma$ induces an \emph{inner} automorphism of $\Pbf\Gbf(\Q_p)$ (that is, conjugation by $\sigma_p$), thus we must have
		\[
		g=f=1.
		\]
		In particular, the automorphism $\psi_\sigma$ of $\Pbf\Gbf(\Q)$ is inner, that is, there exists $C\in\Gbf(\Q)$ such that
		\[
		\sigma_ph\sigma_p^{-1}=ChC^{-1},
		\quad
		\forall
		h\in\Pbf\Gbf(\Q).
		\]
		It follows immediately that $\sigma\in Z_{\Gbf(\Q_p)}\Gbf(\Q)=\Q_p^\times\Gbf(\Q)$, which proves the claim. In fact, we have proved the following
		\begin{equation}\label{Comm(Gamma)}
			\Ccal_{\bfGSp_L(\mathbb{Q}_p)}(\Gamma)
			=
			\mathbb{Q}_p^\times\bfG(\mathbb{Q}).
		\end{equation}

		Next we prove the `if' part: suppose $\sigma\in Z_{\bfG_{\uA}(\Q_p)}\Gbf(\Q)$, then $\sigma\in \Gbf(\Q)\Q_p^\times$ and by definition, we have
		\[
		\Gamma
		=
		\Gbf(\Q)\bigcap\Gbf(\widehat{\Z}^{(p)}),
		\quad
		\Gamma_\sigma
		=
		\sigma^{-1}(\Gbf(\Q)\bigcap\Gbf(\widehat{\Z}^{(p)}))\sigma
		=
		\Gbf(\Q)
		\bigcap
		\sigma^{-1}\Gbf(\widehat{\Z}^{(p)})\sigma.
		\]
		Since $\Gbf(\widehat{\Z}^{(p)})$ and $\sigma^{-1}\Gbf(\widehat{\Z}^{(p)})\sigma$ are both compact and open subgroups in $\Gbf(\widehat{\Z}^{(p)})$, they are thus commensurable. It follows that $\Gamma$ and $\Gamma_\sigma$ are also commensurable.
	\end{proof}

	\begin{theorem}\label{non-commensurable for different primes}
		For distinct $\ell_1,\ell_2\in\mathcal{S}$ and any $\sigma_i\in\mathcal{T}(\ell_i)$ ($i=1,2$), the subgroups
		$\Gamma_p(\ell_1)_{\sigma_1}$ and $\Gamma_p(\ell_2)_{\sigma_2}$ of $\bfGSp_L(\mathbb{Q}_p)$ are not commensurable.
	\end{theorem}
	\begin{proof}
		To simplify notations, we write ($i=1,2$)
		\[
		\bfG_i
		=
		\bfG_{\uA_{0,\ell_i}},
		\quad
		\bfG_i^1
		=
		\bfG_{\uA_{0,\ell_i}}^1,
		\quad
		B_i=B(\ell_i).
		\]

		If $\Gamma_p(\ell_1)_{\sigma_1}$ and $\Gamma_p(\ell_2)_{\sigma_2}$ are commensurable, then they have the same commensurators inside $\bfGSp_L(\mathbb{Q}_p)$. From (\ref{Comm(Gamma)}), we deduce
		\[
		\Ccal_{\bfGSp_L(\mathbb{Q}_p)}(\Gamma_p(\ell_?)_{\sigma_?})
		=
		\mathbb{Q}_p^\times\sigma_?^{-1}
		\bfG_?(\mathbb{Q})\sigma_?
		\quad
		\text{where }
		?=1,2.
		\]

		We next show that for $\ell_1$ and $\ell_2$ distinct,  $\mathbb{Q}_p^\times\sigma_1^{-1}\bfG_1(\mathbb{Q})\sigma_1$ and 
		$\mathbb{Q}_p^\times\sigma_2^{-1}\bfG_2(\mathbb{Q})\sigma_2$ are also distinct: otherwise, take any $g_1\in \bfG_1(\mathbb{Q})$, then it can be written in the form
		\[
		\sigma_1^{-1}g_1\sigma_1=b\sigma_2^{-1}g_2\sigma_2,
		\]
		for some $b\in\mathbb{Q}_p^\times$ and $g_2\in \bfG_2(\mathbb{Q})$.		
		Write $\mathrm{Trd}$ for the reduced trace on $B_1$ and on $B_2$. Then we have
		\[
		\mathrm{Trd}
		(\mathrm{Tr}(g_i))
		=
		\mathrm{Trd}
		(\mathrm{Tr}(\sigma_i^{-1}\tau_{\ell_i}(g_i)\sigma_i)),
		\quad
		i=1,2.
		\]
		If $\mathrm{Trd}(\mathrm{Tr}(g_1))\ne0$, then $\mathrm{Trd}(\mathrm{Tr}(b))=2nb\in\mathbb{Q}^\times$, and thus $b\in\mathbb{Q}^\times$. Therefore
		\[
		\sigma_1^{-1}g_1\sigma_1
		\in
		\sigma_2^{-1}\bfG_2(\mathbb{Q})\sigma_2.
		\]		
		Now for any $g_1\in \bfG_1(\mathbb{Q})$, we can always find $g_1'\in \bfG_1(\mathbb{Q})$ such that $\mathrm{Trd}(\mathrm{Tr}(g_1g_1'))\neq0$ and $\mathrm{Trd}(\mathrm{Tr}(g_1'))\ne0$. One deduces that $\sigma_1^{-1}g_1\sigma_1\in\sigma_2^{-1}\bfG_2(\mathbb{Q})\sigma_2$ for any $g_1\in \bfG_1(\mathbb{Q})$. From this, it follows
		\[
		\sigma_1^{-1}\bfG_1(\mathbb{Q})\sigma_1
		=
		\sigma_2^{-1}\bfG_2(\mathbb{Q})\sigma_2.
		\]
		On the other hand, recall that $\Gbf_1$ is split at $\ell_2$ and non-split at $\ell_1$, while $\Gbf_2$ is split at $\ell_1$ and non-split at $\ell_2$. So there exists an element $h_1\in \sigma_1^{-1}\Gbf_1(\Q)\sigma_1\subset\sigma_1^{-1}\Gbf_1(\Q_{\ell_1})\sigma_1$ such that there is at least one eigenvalue of $h_1$ lies in $\overline{\Q}_{\ell_1}\backslash\Q_{\ell_1}$. But as $\Gbf_2$ is split at $\ell_1$, there does not exist such an element in $\sigma_2^{-1}\Gbf_2(\Q)\sigma_2$. This is a contradiction. Thus we must have $\Q_p^\times\sigma_1^{-1}\Gbf_1(\Q)\sigma_1\neq\Q_p^\times\sigma_2\Gbf_2(\Q)\sigma_2$, which proves the theorem.
	\end{proof}

	\section{Proof of Theorems \ref{surjectivity of Pi_R} and \ref{surjectivity of Pi_R^N}}\label{Proof of the main result}
	For self-containment,
	we reproduce part of \cite[§§3.6 and 3.7]{Cornut2002}. For a group $G$ and a non-empty finite set $\mathcal{T}$, we write the diagonal map
	\[
	\Delta\colon G\rightarrow\prod_{\sigma\in\mathcal{T}}G.
	\]
	For any $\sigma_0\in\mathcal{T}$, write
	\[
	\mathrm{pr}_{\sigma_0}\colon\prod_{\sigma\in\mathcal{T}}G\rightarrow
	G
	\]
	for the map of projecting to the $\sigma_0$-th component. For a non-empty subset $\mathcal{T}'$ of $\mathcal{T}$, we define a subgroup of $\prod_{\sigma\in\mathcal{T}}G$
	\[
	G^{\mathcal{T}'}
	=
	\left\{
	(g_{\sigma})
	\in
	\prod_{\sigma\in\mathcal{T}}
	\mid
	g_{\sigma}=1,
	\,
	\forall
	\sigma\notin\mathcal{T}'
	\right\}.
	\]

	A subgroup $H$ of $G^{\mathcal{T}}$ is called a \textit{product of diagonals} if there are \textit{disjoint} non-empty subsets $\mathcal{T}_1,\cdots,\mathcal{T}_r$ of $\mathcal{T}$ such that $H=\prod_{i=1}^r\Delta^{\mathcal{T}_i}(G)$ (we do not require $\Tcal=\bigsqcup_{i=1}^r\Tcal_i$). Then we have 
	\begin{proposition}\label{product of diagonals}
		Suppose that $G$ is simple and non-commutative. Then a subgroup $H$ of $G^{\mathcal{T}}$ is normalized by $\Delta^{\mathcal{T}}(G)$ if and only if it is a product of diagonals.
	\end{proposition}
	\begin{proof}
		This is \cite[Proposition 3.10]{Cornut2002}.
	\end{proof}

	\subsection{$G$ is simple}\label{G simple}
	From now on, we assume that $G$ is a \emph{simple and non-commutative} $p$-adic Lie group which is generated by one-parameter adjoint unipotent subgroups
	(in the sense of \cite{Ratner1995}). Let $(\Gamma_{\sigma})_{\sigma\in\mathcal{T}}$ be a finite set of discrete and cocompact subgroups of $G$. We write $\Gamma=\prod_{\sigma\in\mathcal{T}}\Gamma_{\sigma}$ and so $\Gamma\backslash G^{\mathcal{T}}$ is compact. We define then a partition of the finite set $\mathcal{T}$:
	\begin{equation}\label{partition of T}
		\mathcal{T}=\bigsqcup_{i=1}^r\mathcal{T}_i,
	\end{equation}
	such that $\sigma,\sigma'\in\mathcal{T}_i$ if and only if $\Gamma_{\sigma}$ and $\Gamma_{\sigma'}$ are commensurable (this is well-defined since commensurability is an equivalence relation).
	\begin{lemma}\label{Gamma_sigma and Gamma_sigma' comm or not}
		Take two elements $\sigma,\sigma'\in\mathcal{T}$. Write $\Gamma_{\sigma,\sigma'}=\Gamma_{\sigma}\times\Gamma_{\sigma'}$ and $\Delta\colon G\rightarrow G^2$ for the diagonal map. Then the closure of $\Gamma_{\sigma,\sigma'}\Delta(G)$ in $G^2$ is
		\[
		\begin{cases*}		
			\Gamma_{\sigma,\sigma'}\Delta(G),
			&
			if $\Gamma_{\sigma}$ and $\Gamma_{\sigma'}$ are commensurable;
			\\
			G^2,
			&
			if $\Gamma_{\sigma}$ and $\Gamma_{\sigma'}$ are not commensurable.
		\end{cases*}
		\]
	\end{lemma}
	\begin{proof}
		By assumption, $G$ is generated by one-parameter adjoint unipotent subgroups, thus we can apply Ratner's theorem on unipotent flows (\cite[Theorem 2]{Ratner1995}) to get that the closure of $\Gamma_{\sigma,\sigma'}\Delta(G)$ in $G^2$ is of the form $\Gamma_{\sigma,\sigma'}H$ for some closed subgroup $H$ of $G^2$. Clearly $H$ is normalized by $\Delta^{\mathcal{T}}(G)$ and thus by Proposition \ref{product of diagonals}, $H$ is either $\Delta(G)$ or $G^2$ (there are only two partitions of the set $\{\sigma,\sigma'\}$).

		Observe that the following natural bijection is a homeomorphism
		\[
		\left(
		\Gamma_{\sigma}\bigcap\Gamma_{\sigma'}
		\right)
		\backslash G
		\rightarrow
		\left(
		\Delta(G)\bigcap\Gamma_{\sigma,\sigma'}
		\right)
		\backslash\Delta(G).
		\]
		On the other hand, since $\Gamma_{\sigma,\sigma'}$ is discrete in $G^2$, $\Delta(G)$ is open in $\Gamma_{\sigma,\sigma'}\Delta(G)$ and the following bijection is again a homeomorphism (because it is an open continuous map)
		\[
		\left(
		\Delta(G)\bigcap\Gamma_{\sigma,\sigma'}
		\right)
		\backslash\Delta(G)
		\rightarrow
		\Gamma_{\sigma,\sigma'}\backslash\Gamma_{\sigma,\sigma'}\Delta(G).
		\]
		In particular, $(\Gamma_{\sigma}\bigcap\Gamma_{\sigma'})\backslash G$ is compact if and only if
		$\Gamma_{\sigma,\sigma'}\backslash\Gamma_{\sigma,\sigma'}\Delta(G)$ is compact, if and only if
		$\Gamma_{\sigma,\sigma'}\backslash\Gamma_{\sigma,\sigma'}\Delta(G)$ is closed in 
		$\Gamma_{\sigma,\sigma'}\backslash G^2$.

		Now if $\Gamma_{\sigma}$ and $\Gamma_{\sigma'}$ are commensurable, then $(\Gamma_{\sigma}\bigcap\Gamma_{\sigma'})\backslash G$ is compact, so $\Gamma_{\sigma,\sigma'}\backslash\Gamma_{\sigma,\sigma'}\Delta(G)$ is closed in 
		$\Gamma_{\sigma,\sigma'}\backslash G^2$, thus $H=\Delta(G)$.
		If they are not commensurable, then $(\Gamma_{\sigma}\bigcap\Gamma_{\sigma'})\backslash G$ is not compact, so $\Gamma_{\sigma,\sigma'}\backslash\Gamma_{\sigma,\sigma'}\Delta(G)$ is not closed in $G^2$, therefore we must have $H=G^2$.
	\end{proof}

	\begin{proposition}\label{closure of diagonal map}
		The closure of $\Gamma\Delta^{\mathcal{T}}(G)$ in $G^{\mathcal{T}}$ is of the form $\Gamma H$ for $H$ a subgroup which is a product of diagonals $\prod_{i=1}^r\Delta^{\mathcal{T}_i}(G)$ for the partition
		$\Tcal=\bigsqcup_{i=1}^n\Tcal_i$.
		In particular, if each $\mathcal{T}_i$ contains only one element, then $\Gamma\Delta^{\mathcal{T}}(G)$ is dense in $G^{\mathcal{T}}$.
	\end{proposition}
	\begin{proof}
		As in the previous lemma, we know that the closure of $\Gamma\Delta^\Tcal(G)$ in $G^\Tcal$ is of the form $\Gamma H$ for some closed subgroup $H$ of $G^{\mathcal{T}}$ of the form
		$H=\prod_{i=1}^{r'}\Delta^{\mathcal{T}_i'}(G)$
		for some \textit{partition} of $\mathcal{T}$:
		\[
		\mathcal{T}=\bigsqcup_{i=1}^{r'}\mathcal{T}_i'.
		\]

		Now take any two elements $\sigma,\sigma'\in\mathcal{T}$ and consider the projection to the components of $\sigma$ and $\sigma'$:
		\[
		\mathrm{pr}_{\sigma}
		\times
		\mathrm{pr}_{\sigma'}
		\colon
		G^{\mathcal{T}}
		\rightarrow
		G^2
		\]
		Clearly the image of $\Gamma H$ under this projection map is equal to the closure of $\Gamma_{\sigma,\sigma'}\Delta(G)$. It follows immediately from the previous lemma that $\sigma,\sigma'\in\mathcal{T}_i$ if and only if the closure of $\Gamma_{\sigma,\sigma'}\Delta(G)$ is \textit{not} equal to $G^2$, if and only if $\Gamma_{\sigma}$ and $\Gamma_{\sigma'}$ are commensurable, if and only if $\sigma,\sigma'\in\mathcal{T}_{i'}$ for some $i'=1,\cdots,r$. Therefore each $\mathcal{T}_i'$ is contained in some $\mathcal{T}_{i'}$. So these two partitions $\mathcal{T}=\bigsqcup_{i=1}^r\mathcal{T}_i$ and $\mathcal{T}=\bigsqcup_{i=1}^{r'}\mathcal{T}_i'$ are the same, which finishes the proof.
	\end{proof}

	We write the quotient group
	\[
	\bfPSp_L(\mathbb{Q}_p)
	=
	\bfSp_L(\mathbb{Q}_p)
	/
	Z(\bfSp_L(\mathbb{Q}_p)),
	\]
	where $Z(\bfSp_L(\mathbb{Q}_p))=Z_{\bfSp_L}(\mathbb{Q}_p)=\{\pm1_{2n}\}$
	is the center of $\bfSp_L(\mathbb{Q}_p)$. Similarly, for a subgroup $H$ of
	$\bfSp_L(\mathbb{Q}_p)$, we write $\mathbf{P}H$	for the image of $H$ under the projection map $\bfSp_L(\mathbb{Q}_p)\rightarrow	\bfPSp_L(\mathbb{Q}_p)$.

	\begin{proof}[Proof of Theorem \ref{surjectivity of Pi_R}]
		Since each $\Gamma_p^1(\ell)_\sigma$ contains $\pm1$, we have a natural bijection
		\[
		\Gamma_p^1(\ell)_{\sigma}\backslash\bfSp_L(\mathbb{Q}_p)
		\simeq
		\mathbf{P}\Gamma_p^1(\ell)_{\sigma}\backslash\bfPSp_L(\mathbb{Q}_p).
		\]
		So it suffices to prove the theorem for $\bfPSp_L(\mathbb{Q}_p)$ instead of $\bfSp_L(\mathbb{Q}_p)$.
		Note that $\mathbf{P}\bfSp_L(\mathbb{Q}_p)$
		is a simple and non-commutative group and it is generated by one-parameter adjoint unipotent subgroups
		(in the sense of \cite{Ratner1995}). Moreover, each $\Gamma_p^1(\ell)_{\sigma}$ is discrete and cocompact
		in $\bfSp_L(\mathbb{Q}_p)$.	Write
		\[
		\Gamma=
		\prod_{\ell\in\mathcal{S},\sigma\in\mathcal{T}_{\ell}}\Gamma_p^1(\ell)_{\sigma},
		\quad
		\mathbf{P}\Gamma
		=
		\prod_{\ell\in\mathcal{S},\sigma\in\mathcal{T}_{\ell}}
		\mathbf{P}\Gamma_p^1(\ell)_{\sigma}.
		\]
		Using Proposition \ref{closure of diagonal map}, we know that the closure of the image of the following simultaneous projection map
		\[
		\bfPSp_L(\mathbb{Q}_p)
		\rightarrow
		\prod_{\ell\in\mathcal{S},\sigma\in\mathcal{T}_{\ell}}
		\mathbf{P}\Gamma_p^1(\ell)_{\sigma}
		\backslash
		\bfPSp_L(\mathbb{Q}_p)
		=
		\mathbf{P}\Gamma
		\backslash
		\prod_{\ell\in\mathcal{S},\sigma\in\mathcal{T}_{\ell}}
		\bfPSp_L(\mathbb{Q}_p)
		\]
		is of the form
		$\mathbf{P}\Gamma\backslash\mathbf{P}\Gamma H$ for a closed subgroup $H$ of $\prod_{\ell\in\mathcal{S},\sigma\in\mathcal{T}_{\ell}}\bfPSp_L(\mathbb{Q}_p)$ which is a product of diagonals by Theorems \ref{commensurable} and \ref{non-commensurable for different primes}
		\[
		H
		=
		\prod_{\ell\in\mathcal{S}}
		\prod_{i=1}^{r(\ell)}
		\Delta^{\mathcal{T}_{\ell,i}}(\bfPSp_L(\mathbb{Q}_p)).
		\]
		From this one deduces easily the theorem.
	\end{proof}

	\begin{proof}[Proof of Theorem \ref{surjectivity of Pi_R^N}]
		We prove the theorem assuming that each $\mathcal{T}_{\ell,i}$ contains only one element. The general case is a combination of this simplified case and the preceding proof.

		Choose a finite set $\mathcal{N}$ of primes $q\nmid N$ such that for any distinct $\sigma_1,\sigma_2\in\mathcal{T}_\ell$ for some $\ell\in\mathcal{S}$, there is $q\in\mathcal{N}$ so that the $q$-th component $(\sigma_1^{-1}\sigma_2)_q$ of $\sigma_1^{-1}\sigma_2$ does \emph{not} lie in $Z_{\bfG_{\uA_0}}(\mathbb{Q})$.
		To prove the theorem, it suffices to prove that for any such $\mathcal{N}$, the image of the following simultaneous \emph{projection} map is dense
		\begin{equation}\label{PSp_L(A_N) surjects onto simultaneous projection}
			\bfPSp_L(\A_{(\Ncal)})
			\to
			\prod_{\ell\in\mathcal{S},\sigma\in\mathcal{T}_{\ell}}
			\mathbf{P}\Gamma_{\mathcal{N}}^1(\ell)_\sigma
			\backslash
			\bfPSp_L(\A_{(\Ncal)}).
		\end{equation}
		Here $\Gamma_{\mathcal{N}}^1(\ell)_\sigma
		=\sigma^{-1}\bfG_{\uA_{0}^{(\ell)}}^1(\Z_{(\mathcal{N})})\sigma$.
		Note that we have the following commutative diagram (here $p\in\mathcal{N}$)
		\[
		\begin{tikzcd}
			\bfPSp_L(\mathbb{Q}_p)
			\arrow[r]
			\arrow[d,hookrightarrow]
			&
			\prod_{\ell\in\mathcal{S},\sigma\in\mathcal{T}_{\ell}}
			\mathbf{P}\Gamma_p^1(\ell)_\sigma
			\backslash
			\bfPSp_L(\mathbb{Q}_p)
			\arrow[d,hookrightarrow]
			\\
			\bfPSp_L(\A_{(\Ncal)})
			\arrow[r]
			\arrow[d,hookrightarrow]
			&
			\prod_{\ell\in\mathcal{S},\sigma\in\mathcal{T}_{\ell}}
			\mathbf{P}\Gamma_{\mathcal{N}}^1(\ell)_\sigma
			\backslash
			\bfPSp_L(\A_{(\Ncal)})
			\arrow[d,hookrightarrow]
			\\
			\bfPSp_L(\A^{(N)}_f)
			\arrow[r]
			&
			\prod_{\ell\in\mathcal{S},\sigma\in\mathcal{T}_{\ell}}
			\mathbf{P}\Gamma^{(N),1}(\ell)_\sigma
			\backslash
			\bfPSp_L(\A^{(N)}_f)
		\end{tikzcd}
		\]

		Write $\mathbf{P}\Gamma
		:=\prod_{\ell\in\mathcal{S},\sigma\in\mathcal{T}_{\ell}}
		\mathbf{P}\Gamma_{\mathcal{N}}^1(\ell)_\sigma$.
		Again, by \cite{Ratner1995}, we know that the closure is of the form
		$\mathbf{P}\Gamma\backslash\mathbf{P}\Gamma H$ for some closed subgroup $H$ of $\prod_{\ell\in\mathcal{S},\sigma\in\mathcal{T}_{\ell}}\bfPSp_L(\A_{(\Ncal)})$.

		For two distinct pairs $(\ell_1,\sigma_1)$ and $(\ell_2,\sigma_2)$, by assumption on $\mathcal{N}$ and Theorem \ref{surjectivity of Pi_R}, there is a prime $p\in\mathcal{N}$ such that the image of the following projection map is dense
		\[
		\bfPSp_L(\mathbb{Q}_p)
		\rightarrow
		\mathbf{P}\Gamma_p^1(\ell_1)_{\sigma_1}
		\backslash
		\bfPSp_L(\mathbb{Q}_p)
		\times
		\mathbf{P}\Gamma_p^1(\ell_2)_{\sigma_2}
		\backslash
		\bfPSp_L(\mathbb{Q}_p).
		\]
		On the other hand, the strong approximation for $\bfSp_L$ implies that $\bfPSp_L(\mathbb{Q}_p)$ is dense in the quotient $\mathbf{P}\Gamma^{(N),1}(\ell)_{\sigma}
		\backslash\bfPSp_L(\A^{(N)}_f)$, in particular it is dense in $\mathbf{P}\Gamma_{\mathcal{N}}^1(\ell)_\sigma
		\backslash
		\bfPSp_L(\A_{(\Ncal)})$. Thus the closed subgroup $H$ contains the subgroup of
		$\prod_{\ell\in\mathcal{S},\sigma\in\mathcal{T}_{\ell}}\bfPSp_L(\A_{(\Ncal)})$ consisting of $(g_{\ell,\sigma})_{\ell\in\mathcal{S},\sigma\in\mathcal{T}_{\ell}}$ such that $g_{\ell,\sigma}=1$ for any $(\ell,\sigma)\neq(\ell_1,\sigma_1),(\ell_2,\sigma_2)$. It follows immediately that $H$ is equal to
		$\prod_{\ell\in\mathcal{S},\sigma\in\mathcal{T}_{\ell}}\bfPSp_L(\A_{(\Ncal)})$. This proves that the image of the map (\ref{PSp_L(A_N) surjects onto simultaneous projection}) is dense.
	\end{proof}

	\subsection{$G$ is not simple}\label{G not simple}
	For later use in §\ref{Shimura varieties of Hodge type}, we also need the case $G$ not simple, but rather a product of simple groups.
	We write in this subsection
	\[
	G=\prod_{j=1}^tG^{(j)},
	\]
	with each $G^{(j)}$ simple and non-commutative.

	\begin{proposition}\label{H is a product of diagonals}
		A subgroup $H$ of $G^\Tcal$ is normalized by $\Delta^\Tcal(G)$ if and only if it is of the form
		\[
		H=\prod_{j=1}^t\prod_{i=1}^{r_j}\Delta^{\Tcal_i^{(j)}}(G^{(j)}),
		\]
		where for each $j=1,\cdots,t$, $\Tcal_1^{(j)},\cdots,\Tcal_{r_j}^{(j)}$ are disjoint non-empty subsets of $\Tcal$.
	\end{proposition}
	\begin{proof}
		For $j=1,\cdots,t$, we write $H^{(j)}=H\bigcap(G^{(j)})^\Tcal$, a subgroup of $(G^{(j)})^\Tcal$. Since $H$ is normalized by $\Delta^\Tcal(G)$, $H^{(j)}$ is normalized by $\Delta^\Tcal(G^{(j)})$. Thus by Proposition \ref{product of diagonals}, we know that $H^{(j)}$ is a product of diagonals in $(G^{(j)})^\Tcal$:
		\[
		H^{(j)}=\prod_{i=1}^{r_j}\Delta^{\Tcal_i^{(j)}}(G^{(j)}).
		\]
		Here $\Tcal_1^{(j)},\cdots,\Tcal_{r_j}^{(j)}$ are disjoint non-empty subsets of $\Tcal$.
		We set $\widetilde{H}=\prod_{j=1}^tH^{(j)}$ and for any $h\in \widetilde{H}$, we write it as
		$h=h_1\cdots h_t$ with $h_j\in(G^{(j)})^\Tcal$. We fix $j=1,\cdots,t$. Then for any $g_j\in\Delta^\Tcal(G^{(j)})$, by assumption, we have
		\[
		g_jhg_j^{-1}
		=
		h_1\cdots h_{j-1}(g_jh_jg_j^{-1})h_{j+1}\cdots h_t\in H.
		\]
		Thus $g_jhg_j^{-1}h^{-1}=g_jh_jg_j^{-1}h_j^{-1}\in H\bigcap(G^{(j)})^\Tcal=H^{(j)}$ for any $g_j\in\Delta^\Tcal(G^{(j)})$. We claim that $h_j\in H^{(j)}$, which then implies $H=\widetilde{H}$ and thus finishes the proof.
		
		Now we prove the claim. We write $h_j=(h_{j,\sigma})_{\sigma\in\Tcal}$ with $h_{j,\sigma}\in G^{(j)}$ and consider the following two cases:
		\begin{enumerate}
			\item 
			Fix $\sigma,\sigma'\in\Tcal_i^{(j)}$, then for any $g\in G^{(j)}$, we have
			\[
			gh_{j,\sigma}g^{-1}h_{j,\sigma}^{-1}
			=
			gh_{j,\sigma'}g^{-1}h_{j,\sigma'}^{-1}.
			\]
			In particular, $h_{j,\sigma}^{-1}h_{j,\sigma'}$ commutes with any $g\in G^{(j)}$. Since $G^{(j)}$ is simple, we must have
			\[
			h_{j,\sigma}=h_{j,\sigma'}.
			\]

			\item 
			Fix $\sigma\notin\bigsqcup_{i=1}^{r_j}\Tcal_i^{(j)}$, then for any $g\in G^{(j)}$, we have
			\[
			gh_{j,\sigma}g^{-1}h_{j,\sigma}^{-1}=1.
			\]
			Thus we have
			\[
			h_{j,\sigma}=1.
			\]
		\end{enumerate}
		From the above two cases, we deduce that $h_j=(h_{j,\sigma})_{\sigma\in\Tcal}$ lies in $H^{(j)}=\prod_{i=1}^{r_j}\Delta^{\Tcal_i^{(j)}}(G^{(j)})$.		
	\end{proof}

	As in Lemma \ref{Gamma_sigma and Gamma_sigma' comm or not}, we have
	\begin{lemma}
		Take two elements $\sigma,\sigma'\in\Tcal$. Then $\Gamma_{\sigma,\sigma'}\Delta(G)$ is closed in $G^2$ if and only if $\Gamma_\sigma$ and $\Gamma_{\sigma'}$ are commensurable.
	\end{lemma}
	\begin{proof}
		The proof is exactly the same as in Lemma \ref{Gamma_sigma and Gamma_sigma' comm or not}, except that now we do not know whether the closure of $\Gamma_{\sigma,\sigma'}\Delta(G)$ is equal to $G^2$ in the case $\Gamma_{\sigma}$ and $\Gamma_{\sigma'}$ not commensurable.
	\end{proof}

	\begin{corollary}\label{Gamma_sigma and Gamma_sigma' comm or not, not simple}
		Suppose that for any $j=1,\cdots,t$, $\Gamma G^{(j)}$ is dense in $G$. Then the closure of $\Gamma_{\sigma,\sigma'}\Delta(G)$ in $G^2$ is
		\[
		\begin{cases*}
			\Gamma_{\sigma,\sigma'}\Delta(G),
			&
			if $\Gamma_{\sigma}$ and $\Gamma_{\sigma'}$ are commensurable;
			\\
			G^2,
			&
			if $\Gamma_{\sigma}$ and $\Gamma_{\sigma'}$ are not commensurable.
		\end{cases*}
		\]
	\end{corollary}
	\begin{proof}
		It suffices to consider the case $\Gamma_{\sigma}$ and $\Gamma_{\sigma'}$ not commensurable.
		Since $\Gamma G^{(j)}$ is dense in $G$, so is $\Gamma_{\sigma}G^{(j)}$ and $\Gamma_{\sigma'}G^{(j)}$.
		By the proof of Lemma \ref{Gamma_sigma and Gamma_sigma' comm or not}, the closure of $\Gamma_{\sigma,\sigma'}\Delta(G)$ in $G^2$ is of the form $\Gamma_{\sigma,\sigma'}H$ for a closed subgroup $H$ of $G^2$ containing $\Delta(G)$ and normalized by $\Delta(G)$. Thus by Proposition \ref{H is a product of diagonals}, $H$ is of the form
		\[
		H=\prod_{j=1}^t\prod_{i=1}^{r_j}\Delta^{\Tcal_i^{(j)}}(G^{(j)}).
		\]
		Here $\Tcal=\{\sigma,\sigma'\}$. We know that $H\supset\Delta(G)$, thus $\Tcal=\bigsqcup_{i=1}^{r_j}\Tcal_i^{(j)}$ for each $j$. Since $\Gamma_{\sigma}$ and $\Gamma_{\sigma'}$ are not commensurable, $H\neq\Delta(G)$. Thus there is some $j_0=1,\cdots,t$ such that $r_j=2$. In other words, $H$ contains $(G^{(j)})^{\Tcal}$. However, by assumption, $\Gamma_{\sigma,\sigma'}(G^{(j)})^{\Tcal}$ is dense in $G^2$, so we deduce that $H=G^2$, which finishes the proof.
	\end{proof}

	Then just as in Proposition \ref{closure of diagonal map}, we have
	\begin{proposition}\label{closure of diagonal map, not simple}
		Suppose that for any $j=1,\cdots,t$, $\Gamma G^{(j)}$ is dense in $G$.
		Then the closure of $\Gamma\Delta^\Tcal(G)$ in $G^\Tcal$ is of the form $\Gamma H$ for a closed subgroup $H$ of $G$ given by
		\[
		H=\prod_{i=1}^r\Delta^{\Tcal_i}(G)
		\]
		for the partition $\Tcal=\bigsqcup_{i=1}^r\Tcal_i$ as in (\ref{partition of T}).

		In particular, if each $\Tcal_i$ contains only one element, then $\Gamma\Delta^{\Tcal}(G)$ is dense in $G^\Tcal$.
	\end{proposition}
	\begin{proof}
		The proof is the same as Proposition \ref{closure of diagonal map}, taking into account of Corollary \ref{Gamma_sigma and Gamma_sigma' comm or not, not simple}.
	\end{proof}

	\section{Application to Shimura varieties of Hodge type}\label{Shimura varieties of Hodge type}
	In this section, we will prove similar results for Shimura varieties of Hodge type as in previous sections, with appropriate modifications of the arguments.

	We first fix some notations, following closely
	\cite{Kisin2010,Kisin2017,Wortmann2013}.
	We fix a Shimura datum of Hodge type $(G,X)$ where $G$ is a connected reductive group over $\Q$. Then we fix a symplectic embedding
	\[
	(G,X)\hookrightarrow(\bfGSp_{L,\Q},S^\pm),
	\]
	where $S^\pm$ is the Siegel upper half space of rank equal to $\mathrm{rk}(L)$. We can and will assume $\mathrm{rk}(L)\ge4$ in the following.
	We fix a compact open subgroup $U$ of $G(\A_f)$ and write $ Sh_{U,\C}= Sh_{U,\C}(G,X)$ for the Shimura variety attached to the triple $(G,X,U)$ whose $\C$-points are given by
	\begin{equation}\label{S_K(C) and double coset}
		Sh_{U,\C}(\C)
		=
		G(\Q)\backslash(X\times G(\A_f)/U),
	\end{equation}
	which has a canonical model defined over the reflex field $F=F(G,X)$. We denote this canonical model by $ Sh_{U,F}$. Similarly, for a compact open subgroup $\Ubf$ of $\bfGSp_L(\A_f)$, we have a canonical model $ Sh_{\Ubf,\Q}$ over $\Q$. Suppose that $U$ is contained in $\Ubf$, then we have a morphism of Shimura varieties
	\[
	Sh_{U,F}\to Sh_{\Ubf,F}= Sh_{\Ubf,\Q}\times_\Q E.
	\]

	For a place $v$ of $F$ over a rational prime $\ell$ such that $G$ is unramified at $\ell$ and $U=U^\ell U_\ell$ with $U_\ell$ hyperspecial at $\ell$, Kisin (\cite{Kisin2010}) constructed an integral canonical model of $ Sh_{U,F}$ over $\Ocal_{(v)}$, the localization of $\Ocal_F$ at $v$. We denote this integral model by $ Sh_{U,\Ocal_{(v)}}$. We assume in the following that $\Ubf=\Ubf^\ell\Ubf_\ell$ with $\Ubf_\ell$ also hyperspecial at $\ell$.
	We write $\Acal_{\Ubf,\Ocal_{(v)}}\to Sh_{\Ubf,\Ocal_{(v)}}$ for the universal abelian scheme over $ Sh_{\Ubf,\Ocal_{(v)}}$ and denote by
	\[
	\Acal_{K,\Ocal_{(v)}}\to Sh_{U,\Ocal_{(v)}}
	\]
	its pull-back along $ Sh_{U,\Ocal_{(v)}}\to Sh_{\Ubf,\Ocal_{(v)}}$. We fix a reductive model $\Gcal$ of $G$ over $\Z_{(\ell)}$ such that $U_\ell=\Gcal(\Z_\ell)$. Then there is a finite set of tensors $s^{(\ell)}=(s_i^{(\ell)})_i\subset L^\otimes_{\Z_{(\ell)}}$ such that $\Gcal$ is identified with the stabilizer in $\GL(L_{\Z_{(\ell)}})$ of these tensors $s^{(\ell)}$ via our symplectic embedding $G\hookrightarrow\bfGSp_{L,\Q}$ (\cite[§2.3.1 and §2.3.2]{Kisin2010}).\footnote{Here $L_{\Z_{(\ell)}}^\otimes$ is the direct sum of all $\Z_{(\ell)}$-modules arising from $L_{\Z_{(\ell)}}$ by taking a finite number of the following operations: duals, tensor products, symmetric powers and exterior powers.}

	For each point $x\in Sh_{U,\Ocal_{(v)}}(\overline{\Q})$, Kisin constructed tensors $s^{(\ell)}_{\ell,x}=(s^{(\ell)}_{i,\ell,x})_i\subset(H_\etrm^1(\Acal_x,\Z_\ell)\otimes_{\Z_\ell}B_\crys)^\otimes$ (\cite[§1.3.6]{Kisin2017}). Here $\Acal_x$ is the specialization of $\Acal_{K,\Ocal_{(v)}}$ to the point $x$. For each prime $\ell'\neq\ell$, Kisin also constructed tensors $s^{(\ell)}_{\ell',x}=(s^{(\ell)}_{i,\ell',x})\subset(H^1_\etrm(\Acal_x,\Q_{\ell'}))^\otimes$ (\emph{loc.cit}).

	Similarly, for each point $x\in Sh_{U,\Ocal_{(v)}}(\overline{\F}_\ell)$, Kisin constructed tensors $s^{(\ell)}_{0,x}=(s^{(\ell)}_{i,0,x})_i\subset(H^1_\crys(\Acal_x/W(\overline{\F}_\ell))\otimes_{W(\overline{\F}_\ell)}B_\crys)^\otimes$ (\cite[§1.3.10]{Kisin2017}). For each prime $\ell'\neq\ell$, there are tensors $s^{(\ell)}_{\ell',x}=(s^{(\ell)}_{i,\ell',x})_i\subset H^1_\etrm(\Acal_x,\Q_{\ell'})^\otimes$ (\cite[§1.3.6]{Kisin2017}). Similarly, we have tensors $s^{(\ell)}_{\ell,x}=(s^{(\ell)}_{i,\ell,x})$ (\emph{loc.cit}). We have a natural isomorphism of $W(\overline{\F}_\ell)$, resp. $\Q_{\ell'}$-modules
	\begin{equation}\label{relation between Hodge tensors}
		L_{W(\overline{\F}_\ell)}\simeq
		H^1_\crys(\Acal_x/W(\overline{\F}_\ell)),\,
		\text{resp. }
		L_{\Q_{\ell'}}\simeq
		H^1_\etrm(\Acal,\Q_{\ell'}),
	\end{equation}
	taking the tensors $s^{(\ell)}_i$ to $s^{(\ell)}_{i,0,x}$, resp. $s^{(\ell)}_{i,\ell',x}$ for all $i$. These tensors for $x\in Sh_{U,\Ocal_{(v)}}(\overline{\F}_\ell)$ are constructed by considering a lifting $\widetilde{x}\in Sh_{U,\Ocal_{(v)}}(\overline{\Q})$ of $x$ and using $\ell'$-adic comparison theorems (\cite[Proposition 1.3.7]{Kisin2017}).

	For $x$ either in $ Sh_{U,\Ocal_{(v)}}(\overline{\Q})$ or in $ Sh_{U,\Ocal_{(v)}}(\overline{\F}_\ell)$, we define an algebraic subgroup $I_x$ of $\Aut_\Q(\Acal_x)$, consisting of elements fixing the tensors $s^{(\ell)}_{i,0,x}$ and $s^{(\ell)}_{i,\ell',x}$ for any $\ell'\neq\ell$.

	\begin{lemma}\label{I_x(Q_p)=G(Q_p) for supersingular}
		Let $x\in Sh_{U,\Ocal_{(v)}}(\overline{\F}_\ell)$. If $\Acal_x$ is a supersingular abelian variety, then for any prime $\ell'\neq\ell$, we have an isomorphism of algebraic groups over $\Q_{\ell'}$
		\[
		I_{x/\Q_{\ell'}}:=I_x\times_\Q\Q_{\ell'}\simeq G_{\Q_{\ell'}}.
		\]
	\end{lemma}
	\begin{proof}
		By definition, for any $\Q_{\ell'}$-algebra $R$, $I_{x/\Q_{\ell'}}(R)$ consists of elements in $\Aut_{\Q_{\ell'}}(\Acal_x)(R)$ fixing the tensors $s^{(\ell)}_{i,0,x}$ and $s^{(\ell)}_{i,\ell',x}$. By the assumption that $\mathrm{rk}(L)\ge4$, there exists a supersingular elliptic curve $E(\ell)$ over $\overline{\F}_\ell$ such that $\Acal_x$ is isogenous to $E(\ell)^{\mathrm{rk}(L)/2}$ (by the theorem of Deligne-Ogus-Shioda). So in particular, we have an isomorphism of algebraic groups over $\Q$
		\[
		\Aut_\Q(\Acal_x)
		\simeq
		\GL_{\mathrm{rk}(L)/2}(B(\ell)),
		\]
		where $B(\ell)=\End_{\overline{\F}_\ell}(E(\ell))^\circ$. Since $B(\ell)$ is unramified at $\ell'\neq\ell$, we have an isomorphism $B(\ell)\otimes_\Q\Q_{\ell'}\simeq\mathrm{M}_2(\Q_{\ell'})$ and thus
		\[
		\Aut_{\Q_{\ell'}}(\Acal_x)
		\simeq
		\GL_{\mathrm{rk}(L)/2}(B(\ell)\otimes\Q_{\ell'})
		\simeq
		\GL_{\mathrm{rk}(L)}(\Q_{\ell'}).
		\]

		Since we have the symplectic embedding $G\hookrightarrow\bfGSp_{L,\Q}$, without loss of generality, we can assume that one of tensors, say, $s^{(\ell)}_{i_0}$, corresponds to the symplectic pairing $\langle-,-\rangle$ on $L_{\Z_{(\ell)}}$. So using the isomorphisms in (\ref{relation between Hodge tensors}), the stabilizer in $\Aut_{\Q_{\ell'}}(\Acal_x)$ of the tensor $s^{(\ell)}_{i_0,\ell',x}$ is exactly $\bfGSp_{L/\Q_{\ell'}}$. It follows that the stabilizer in $\Aut_{\Q_{\ell'}}(\Acal_x)$ of all the tensors $(s^{(\ell)}_{i,\ell',x})_i$ is exactly the group $G_{\Q_{\ell'}}$. Thus we have $I_{x/\Q_{\ell'}}\simeq G_{\Q_{\ell'}}$.
	\end{proof}
	We deduce that $I_x$ is a form of $G$. Moreover, if $\Acal_x$ is supersingular, then $I_x(\R)$ is a closed subgroup of $\Gbf_{\Acal_x}(\R)$, which is compact, so it follows that $I_x(\R)$ is also compact.

	As in Definition \ref{p-adic Hecke orbit}, for a prime number $p$ such that $U=U^pU_p$ and a point $x\in Sh_{U,\Ocal_{(v)}}(K)$ where $K$ is an algebraically closed field, we write
	\[
	\Hcal_p^G(\Acal_x)
	\]
	for the set of points $x'\in Sh_{U,\Ocal_{(v)}}(K)$ such that there is an isogeny $\phi\colon\Acal_x\to\Acal_{x'}$ of degree $p^r$ for some $r\in\Z$ sending the tensors $s^{(\ell)}_{0,x}$, resp. $s^{(\ell)}_{\ell',x}$ to the tensors $s^{(\ell)}_{0,x'}$, resp. $s^{(\ell)}_{\ell',x'}$ (for any $\ell'$). For $x\in Sh_{U,\Ocal_{(v)}}(K)$ with $K=\overline{\Q}$ or $\overline{\F}_\ell$, by (\ref{p-adic Hecke orbit of A, char=0}), (\ref{p-adic Hecke orbit of A, char=ell}) and \cite[Corollary 1.4.2 \& Propositions 2.1.3, 2.1.5]{Kisin2017}, we have a natural bijection
	\[
	\Theta_x\colon
	\Hcal_p^G(\Acal_x)
	\simeq
	I_x(\Q)\backslash G(\Q_p)U/U.
	\]
	Note that the right hand side is in natural bijection with $(I_x(\Q)\bigcap U^p)\backslash G(\Q_p)/U_p$.

	For $K=\overline{\Q}$, the point $x\in Sh_{U,\Ocal_{(v)}}(K)\subset Sh_{U,\Ocal_{(v)}}(\C)$ corresponds to an element $(h_x,g_x)\in X\times G(\A_f)$ under the identification (\ref{S_K(C) and double coset}). Then for any $g\in G(\Q_p)$, $\Theta_x^{-1}(g)\in Sh_{U,\Ocal_{(v)}}(K)\subset Sh_{U,\Ocal_{(v)}}(\C)$ corresponds to the element $(h_x,g_xg)\in X\times G(\A_f)$ under (\ref{S_K(C) and double coset}). Moreover, for $x\in Sh_{U,\Ocal_{(v)}}(\overline{\F}_\ell)$ and a lifting $\widetilde{x}\in Sh_{U,\Ocal_{(v)}}(\overline{\Q})$ of $x$, we have the following commutative diagram (\cite[Corollary 1.4.12]{Kisin2017})
	\begin{equation*}
		\begin{tikzcd}
			I_{\widetilde{x}}(\Q)\backslash G(\Q_p)U/U
			\arrow[d]
			&
			\Hcal_p^G(\Acal_{\widetilde{x}})
			\arrow[l,"\Theta_{\widetilde{x}}"',"\simeq"]
			\arrow[r,hookrightarrow]
			\arrow[d,"\Rbf_\ell"]
			&
			Sh_{U,\Ocal_{(v)}}(\Ocal_{\overline{\Q}})
			\arrow[d,"\Rbf_\ell"]
			\\
			I_x(\Q)\backslash G(\Q_p)U/U
			&
			\Hcal_p^G(\Acal_x)
			\arrow[l,"\Theta_x"',"\simeq"]
			\arrow[r,hookrightarrow]
			&
			Sh_{U,\Ocal_{(v)}}(\overline{\F}_\ell)
		\end{tikzcd}
	\end{equation*}
	The left vertical arrow is the natural projection map induced by the inclusion $I_{\widetilde{x}}(\Q)\hookrightarrow I_x(\Q)$.

	We consider a point $x\in Sh_{U,\Ocal_{(v)}}(\overline{\F}_\ell)$ with $\Acal_x$ supersingular, $U_p=I_x(\Z_p)$. We write $I_x^1$ for the derived subgroup of $I_x$ and $\proj\colon I_x\to I_x/I_x^1=:I_x^\ab$ for the projection map, where $I_x^\ab$ is a torus over $\Q$. We assume
	\begin{equation}\label{condition on similitude of I_x}
		\proj(I_x(\Q_p))
		=
		\proj(I_x(\Z_p))\proj(I_x(\Q)\bigcap U^p).
	\end{equation}
	Then the inclusion $I_x^1\hookrightarrow I_x$ induces a bijection
	\begin{equation}\label{I_x and I_x^1}
		(I_x^1(\Q)\bigcap U^p)\backslash I_x^1(\Q_p)/I_x^1(\Z_p)
		\simeq
		(I_x(\Q)\bigcap U^p)\backslash I_x(\Q_p)/I_x(\Z_p).
	\end{equation}
	Since $I_x^1(\Z_p)$ contains the center of $I_x^1(\Q_p)$, we have
	\begin{equation}\label{I_x^1 and PI_x^1}
		(I_x^1(\Q)\bigcap U^p)\backslash\Pbf I_x^1(\Q_p)/\Pbf I_x^1(\Z_p)
		\simeq
		(I_x^1(\Q)\bigcap U^p)\backslash I_x^1(\Q_p)/I_x^1(\Z_p).
	\end{equation}
    We write $I_x^{1,\ad}$ for the adjoint quotient of $I_x^1$. Then $I_x^{1,\ad}$ is the derived subgroup of $I_x^\ad$. By \cite{Humphreys1969}, we know that
	\begin{equation*}
		\Aut(I_x^{\ad}(\Q))=I_x^{\ad}(\Q)/Z_{I_x^{\ad}}(\Q)\rtimes\Aut_\Q(\Dfrak_x^{\ad}),
	\end{equation*}
	where $\Dfrak_x^{\ad}$ is the Dynkin diagram of $I_x^{\ad}$ and $\Aut_\Q(\Dfrak_x^{\ad})$ is the group of automorphisms over $\Q$ of $\Dfrak_x^{1,\ad}$.

	Now we fix a point $x_0\in Sh_{U,F}(\overline{\Q})$ and a prime $p$ such that $U=U^pU_p$ with $U_p$ hyperspecial and $G$ is split at $p$. We then fix a non-empty finite set $\Scal$ of primes $\lfrak$ of $F$ not over $p$ (as well as a place $v(\lfrak)$ of $\overline{\Q}$ over $\lfrak$) such that $x_0$ has \emph{supersingular} good reduction at $v(\lfrak)$ for all $\lfrak\in\Scal$. For each $\lfrak\in\Scal$, we also fix a non-empty finite subset $\Tcal_\lfrak$ of $Z_{I_{x_0}}(\A_f)$ and write
	\[
	\Tcal=\bigsqcup_{\lfrak\in\Scal}\Tcal_\lfrak.
	\]
	Moreover, for each $\lfrak\in\Scal$, we write $x_0^{(\lfrak)}$ for the reduction modulo $v(\lfrak)$ of $x_0$.
	As in §\ref{p-adic Hecke orbit of A}, we write ($\lfrak\in\Scal$, $\sigma\in\Tcal_\lfrak$)
	\begin{align*}
		\Gamma_p(\lfrak)
		&
		:=I_{x_0^{(\lfrak)}}(\Q)\bigcap U^p,
		\quad
		\Gamma_p(\lfrak)_\sigma:=\sigma^{-1}\Gamma_p(\lfrak)\sigma,
		\\
		\Gamma_p^1(\lfrak)
		&
		:=\Gamma_p(\lfrak)\bigcap I_{x_0^{(\lfrak)}}^1(\Q),
		\quad
		\Gamma_p^1(\lfrak)_\sigma:=\sigma^{-1}\Gamma_p^1(\lfrak)\sigma.
	\end{align*}	
	Then we have a partition
	\[
	\Tcal_\lfrak=\bigsqcup_{i=1}^{r(\lfrak)}\Tcal_{\lfrak,i},
	\]
	such that $\sigma,\sigma'\in\Tcal_{\lfrak,i}$ if and only if
	$\Gamma_p(\lfrak)_\sigma$ and $\Gamma_p(\lfrak)_{\sigma'}$ are commensurable. The same argument as in the proof of Theorem \ref{commensurable} gives
	\begin{theorem}\label{commensurable-Hodge case}
		For two elements $\sigma_1,\sigma_2\in\Tcal_\lfrak$, $(\sigma_1\sigma_2^{-1})_p\notin Z_{I_{x_0^{(\lfrak)}}(\Q_p)}I_{x_0^{(\lfrak)}}(\Q)$ if and only if $\Gamma_p(\lfrak)_{\sigma_1}$ and $\Gamma_p(\lfrak)_{\sigma_2}$ are not commensurable.
	\end{theorem}

	We have the following simultaneous reduction map of the $p$-adic Hecke orbit of $\Acal_{x_0}$:
	\[
	\Rbf_{\Scal,\Tcal}^G
	\colon
	\Hcal_p^G(\Acal_{x_0})
	\to
	\prod_{\lfrak\in\Scal}\prod_{\sigma\in\Tcal_\lfrak}\Hcal_p^G(\Acal_{x_0^{(\lfrak)}}),
	\quad
	\Acal_{x_0'}
	\mapsto
	(\Acal_{(\sigma x_0)^{(\lfrak)}})_{\lfrak\in\Scal,\sigma\in\Tcal_\lfrak}.
	\]
	Then we have
	\begin{theorem}\label{R_{S,T}^G surjective}
		We assume the following
		\begin{enumerate}
			\item 
			(\ref{condition on similitude of I_x}) holds;

			\item 
			$G^1_\Q$ is simply connected;

			\item 
			$\Pbf G^1_\Q(\Q_p)$ is simple and non-commutative;

			\item 
			for any distinct $\lfrak,\lfrak'\in\Scal$, the subgroups $I_{x_0^{(\lfrak)}}(\Q)\bigcap U^p$ and $I_{x_0^{(\lfrak')}}(\Q)\bigcap U^p$ of $U_p$ are \emph{not} commensurable. 
		\end{enumerate}
		Then the image $\Rbf^G_{\Scal,\Tcal}(\Hcal_p^G(\Acal_{x_0}))$ is equal to
		\[
		\prod_{\lfrak\in\Scal}\prod_{i=1}^{r(\lfrak)}\widetilde{\Delta}^{\Tcal_{\lfrak,i}}(\Hcal_p^G(\Acal_{x_0^{(\lfrak)}})).
		\]
		Here $\widetilde{\Delta}^{\Tcal_{\lfrak,i}}$ is defined in the same way as (\ref{twisted diagonal map}).
	\end{theorem}

	By (\ref{I_x and I_x^1}) and (\ref{I_x^1 and PI_x^1}), Theorem \ref{R_{S,T}^G surjective} follows from (\emph{cf.} Theorem \ref{surjectivity of Pi_R})
	\begin{theorem}
		Keep the assumptions as in Theorem \ref{R_{S,T}^G surjective}, then the closure of the image of the simultaneous \emph{projection} map
		\[
		\Pbf G^1(\Q_p)
		\to
		\prod_{\lfrak\in\Scal}\prod_{\sigma\in\Tcal_\lfrak}
		\Pbf\Gamma_p^1(\lfrak)_\sigma\backslash\Pbf G^1(\Q_p)
		\]
		is equal to
		\[
		\prod_{\lfrak\in\Scal}\prod_{i=1}^{r(\lfrak)}
		\Pbf\Gamma_p^1(\lfrak)_i\backslash
		\left(
		\Pbf\Gamma_p^1(\lfrak)_i\Delta^{\Tcal_{\lfrak,i}}(\Pbf G^1(\Q_p))
		\right),
		\]
		where $\Pbf\Gamma_p^1(\lfrak)_i
		=\prod_{\sigma\in\Tcal_{\lfrak,i}}\Pbf\Gamma_p^1(\lfrak)_\sigma$ and $\Delta^{\Tcal_{\lfrak,i}}\colon \Pbf G^1(\Q_p)\to \Pbf G^1(\Q_p)^{\Tcal_{\lfrak,i}}$ is the diagonal embedding.
	\end{theorem}
	\begin{proof}
		Note that $G^1(\Q_p)=I_{x_0^{(\lfrak)}}^1(\Q_p)$ for any $\lfrak\in\Scal$. We have the following partition of $\Tcal$
		\[
		\Tcal=
		\bigsqcup_{\lfrak\in\Scal}\bigsqcup_{i=1}^{r(\lfrak)}\Tcal_{\lfrak,i},
		\]
		such that $\sigma,\sigma'\in\Tcal$ lies in the same $\Tcal_{\lfrak,i}$ if and only if $\sigma,\sigma'\in\Tcal_\lfrak$ and $\Gamma_p(\lfrak)_\sigma$ and $\Gamma_p(\lfrak)_{\sigma'}$ are commensurable: if $\sigma,\sigma'\in\Tcal_\lfrak$, we use Theorem \ref{commensurable-Hodge case}; otherwise we use the assumption on $I_{x_0^{(\lfrak)}}(\Q)\bigcap U^p$  and $I_{x_0^{(\lfrak')}}(\Q)\bigcap U^p$.

		Since $G^1$ is simply connected, $\Pbf G^1(\Q_p)$ is generated by one-parameter adjoint unipotent subgroups (in the sense of \cite{Ratner1995}). Moreover, $\Pbf G^1(\Q_p)$ is simple and non-commutative. Thus, by Proposition \ref{closure of diagonal map}, the closure of the image of map in the theorem is of the form $\Pbf\Gamma\backslash\Pbf\Gamma H$ for some closed subgroup $H$ of $\prod_{\lfrak\in\Scal,\sigma\in\Tcal_\lfrak}\Pbf G^1(\Q_p)$. Here $\Gamma=\prod_{\lfrak\in\Scal,\sigma\in\Tcal_\lfrak}\Gamma_p^1(\lfrak)_\sigma$. Moreover, $H$ is a product of diagonals
		\[
		H=\prod_{\lfrak\in\Scal}\prod_{i=1}^{r(\lfrak)}
		\Delta^{\Tcal_{\lfrak,i}}(\Pbf G^1(\Q_p)).
		\]
		This proves the theorem.
	\end{proof}

	\begin{remark}
		We can also formulate a version of Theorem \ref{R_{S,T}^G surjective} for prime-to-$N$ Hecke orbits of $\Acal_{x_0}$, which we leave to the interested reader. The proof is the same as Theorem  \ref{surjectivity of Pi_R^N}.
	\end{remark}

	More generally, we have the following (\emph{cf.} Theorem \ref{R_{S,T}^G surjective})
	\begin{theorem}\label{R_{S,T}, not simple}
		We assume the following
		\begin{enumerate}
			
			\item 
			(\ref{condition on similitude of I_x}) holds;

			\item 
			$G^1$ is simply connected;

			\item 
			there is a decomposition 
			$\Pbf G^1(\Q_p)=\prod_{j=1}^tG^{(j)}$ into subgroups with each $G^{(j)}$ simple and non-commutative;

			\item 
			for distinct $\lfrak,\lfrak'\in\Scal$, the subgroups $I_{x_0^{(\lfrak)}}(\Q)\bigcap U^p$ and $I_{x_0^{(\lfrak')}}(\Q)\bigcap U^p$ are \emph{not} commensurable;

			\item 
			for each $j=1,\cdots,t$ and $\lfrak\in\Scal$, $\Gamma_p(\lfrak)G^{(j)}$ is dense in $\Pbf G^1(\Q_p)$.

		\end{enumerate}
		Then the image $\Rbf_{\Scal,\Tcal}(\Hcal_p^G(\Acal_{x_0}))$ is equal to
		\[
		\prod_{\lfrak\in\Scal}
		\prod_{i=1}^{r(\lfrak)}
		\widetilde{\Delta}^{\Tcal_{\lfrak,i}}(\Hcal_p^G(\Acal_{x_0^{(\lfrak)}})).
		\]
	\end{theorem}
	\begin{proof}
		The proof is the same as Theorem \ref{R_{S,T}^G surjective} by applying Proposition \ref{closure of diagonal map, not simple} instead of Proposition \ref{closure of diagonal map}, except that we need to verify that the assumption in Proposition \ref{closure of diagonal map, not simple} is true, which is ensured by (5) in the theorem.
	\end{proof}
	
	\begin{remark}
		We leave to the interested reader to formulate and prove the prime-to-$N$ version of the the above theorem, just like Theorem \ref{R_{S,T}^N is an isomorphism}.
	\end{remark}

	\section{Galois orbit of CM abelian varieties}\label{CM abelian varieties}
	In this section, we apply Theorem \ref{R_{S,T} is an isomorphism} to CM abelian varieties.
	We first review the main theorems of complex multiplication, following closely \cite{Lang1983}. We refer the reader to \emph{loc.cit} for more details and explanation of some notions below. We start with the case of simple abelian varieties with CM, and then abelian varieties which is isotypic over $\overline{\mathbb{Q}}$, at last we review the theory of CM for a general CM abelian variety. In the end, we give a description of the image of simultaneous supersingular reductions of Galois conjugates of CM abelian varieties.

	\textbf{(1) The case of simple abelian varieties.}
	Let $\uA=(A,\lambda)$ be a \emph{simple} abelian variety over $\overline{\mathbb{Q}}$ of dimension $n$ with CM given by
	\[
	\iota\colon E\simeq\mathrm{End}_{\overline{\mathbb{Q}}}(\uA)^\circ,
	\]
	where $E$ is a CM \emph{field}. In particular $[E\colon\mathbb{Q}]=2n$, $\iota\colon E\simeq\mathrm{End}_{\overline{\mathbb{Q}}}(\uA)^\circ
	$ and $(A,\iota)$ gives rise to a CM type $(E,\Phi)=(E,\Phi_1,\cdots,\Phi_n)$ where $\Phi_i\colon E\to\mathbb{C}$ are pairwise non-conjugate complex embeddings of $E$. We write $(E',\Phi')=(E',\Phi'_1,\cdots,\Phi'_n)$ for the reflex CM type of $(E,\Phi)$, where $E'$ is a subfield of $\mathbb{C}$. Then we have the type norm associated to $\Phi'$
	\[
	\Nm_{\Phi'}
	\colon
	E'\to E,
	\quad
	x\mapsto\prod_{\sigma\in\Phi'}\sigma(x),
	\]
	which induces the type norm on idèles of $E'$
	\[
	\Nm_{\Phi'}
	\colon
	(E')^\times\backslash\A_{E',f}^\times
	\to
	E^\times\backslash\A_{E,f}^\times.
	\]
	By taking the closure $\overline{(E')^\times}$ (resp. $\overline{E^\times}$) of $(E')^\times$ (resp. $E^\times$) inside $\A_{E',f}^\times$ (resp. $\A_{E,f}^\times$), we get
	\[
	\Nm_{\Phi'}
	\colon
	\overline{(E')^\times}\backslash\A_{E',f}^\times
	\to
	\overline{E^\times}\backslash\A_{E,f}^\times.
	\]
	Recall $\bfG_{\uA}(R)$ consists of $g\in(\mathcal{O}\otimes R)^\times$ such that $g'g=\mu(g)\in R^\times$ where $g'=\lambda^{-1}g^\vee\lambda$. The type norm factors through $\bfG_{\uA}(\A_f)$:
	\[
	\Nm_{\Phi'}
	\colon
	\overline{(E')^\times}\backslash\A_{E',f}^\times
	\to
	\overline{\bfG_{\uA}(\mathbb{Q})}
	\backslash
	\bfG_{\uA}(\A_f).
	\]

	We write the Artin reciprocity map as 
	\[
	\Art_{E'}
	\colon
	\Gal((E')^{\mathrm{ab}}/E')
	\simeq
	\overline{(E')^\times}
	\backslash
	\A_{E',f}^\times.
	\]
	Then we denote the composition $\rho_{(E',\Phi')}:=\Nm_{\Phi'}\circ\Art_{E'}$
	\[
	\begin{tikzcd}
		\Gal((E')^\mathrm{ab}/E')
		\arrow[r,"\Art_{E'}"]
		\arrow[rd,"\rho_{(E',\Phi')}"']
		&
		\overline{(E')^\times}\backslash\A_{E',f}^\times
		\arrow[d,"\Nm_{\Phi'}"]
		\\
		&
		\overline{\bfG_{\uA}(\mathbb{Q})}\backslash\bfG_{\uA}(\A_f)
	\end{tikzcd}
	\]
	This map is compatible with the Galois action of $\Gal((E')^{\mathrm{ab}}/E')$ on $\uA$ by the main theorems of CM (\cite[§3.6]{Lang1983}):  more precisely, if we fix an analytic parametrisation
	\[
	\theta\colon\mathbb{C}^n/\Phi(\mathfrak{a})\simeq A(\mathbb{C}),
	\]
	where $\mathfrak{a}$ is a lattice in $E$ ($\Phi\colon E\to\mathbb{C}^n$ is the map sending $x$ to $(\Phi_1(x),\cdots,\Phi_n(x))$), then for any $\sigma\in\Gal((E')^\mathrm{ab}/E')$, we have the following commutative diagram
	\[
	\begin{tikzcd}
		E/\mathfrak{a}
		\arrow[r,"\theta\circ\Phi"]
		\arrow[d,"\rho_{(E',\Phi')}(\sigma^{-1})"']
		&
		A(\overline{\mathbb{Q}})_{\mathrm{tor}}
		\arrow[d,"\sigma"]
		\\
		E/\rho_{(E',\Phi')}(\sigma^{-1})\mathfrak{a}
		\arrow[r,"\theta_\sigma\circ\Phi"]
		&
		\sigma(A)(\overline{\mathbb{Q}})_{\mathrm{tor}}
	\end{tikzcd}
	\]

	\textbf{(2) The case of isotypic abelian varieties.}
	Fix a positive integer $d$ and write $\uA'=\uA^d$. 	We have an embedding
	\[
	\iota'
	\colon
	E^d\hookrightarrow
	\mathrm{End}_{\overline{\mathbb{Q}}}(\uA')
	=
	\mathrm{Mat}_d(\mathrm{End}_{\overline{\mathbb{Q}}}(\uA))
	\]
	and thus $(\uA',\iota')$ gives rise to a CM type $(E^d,\Phi^d)$. We have an embedding
	\[
	\bfG_{\uA}\hookrightarrow\bfG_{\uA'}.
	\]
	One sees easily that $\bfG_{\uA}$ is mapped into the center $Z_{\bfG_{\uA'}}$ of $\bfG_{\uA'}$.

	The reflex type of $(E^d,\Phi^d)$ is $((E')^d,(\Phi')^d)$, so we have the composition map $\rho_{((E')^d,(\Phi')^d)}$ defined by the following commutative diagram
	\[
	\begin{tikzcd}
		\Gal((E')^\mathrm{ab}/E')
		\arrow[d,"\rho_{((E')^d,(\Phi')^d)}"']
		\arrow[rr,"\Art_{E'}"]
		&&
		\overline{(E')^\times}\backslash\A_{E',f}^\times
		\arrow[d,"\mathrm{diag}"]
		\\
		(\overline{E^\times}\backslash\A_{E,f}^\times)^d
		&&
		(\overline{(E')^\times}\backslash\A_{E',f}^\times)^d
		\arrow[ll,"\Nm_{\Phi'}^d=\Nm_{(\Phi')^d}"]
	\end{tikzcd}
	\]
	From this we get the homomorphism
	\[
	\rho_{((E')^d,(\Phi')^d)}
	\colon
	\Gal((E')^\mathrm{ab}/E')
	\rightarrow
	\overline{Z_{\bfG_{\uA'}}(\mathbb{Q})}
	\backslash
	Z_{\bfG_{\uA'}}(\A_f).
	\]
	This map is also compatible with the action of $\Gal((E')^\mathrm{ab}/E')$ on $\uA'$, just as in the case of simple abelian varieties.

	\textbf{(3) The case of general abelian varieties.}
	Suppose $\uA_0=(A_0,\lambda_0,\overline{\eta})$ is a polarized abelian variety with $U$-level structure where $U\subset\Isom_K((A_0,\lambda_0)\pdiv,(A_0,\lambda_0)\pdiv)$ is a compact open subgroup. We assume that the polarized abelian variety $(A_0,\lambda_0)$ is CM, in other words, it is isogenous over $\overline{\mathbb{Q}}$ to $\prod_{i=1}^r\uA_i^{n_i}$, where each $\uA_i$ is a simple polarized abelian variety over $\overline{\mathbb{Q}}$ (not isogenous to each other) which has CM given by $E_i\simeq\mathrm{End}_{\overline{\mathbb{Q}}}(\uA_i)^\circ$ where $E_i$ is a CM field. Then we have
	$\mathrm{End}_{\overline{\mathbb{Q}}}(\uA)^\circ
	\simeq
	\prod_{i=1}^r\mathrm{Mat}_{n_i}(E_i)$.
	Set
	\[
	E_0=\prod_{i=1}^rE_i^{n_i},
	\]
	a semi-simple CM algebra over $\mathbb{Q}$.
	Then the embedding $E_0\hookrightarrow\mathrm{End}_{\overline{\mathbb{Q}}}(\uA_0)^\circ$ gives rise to a CM type $(E_0,\Phi_0)$ with $\Phi_0=\prod_{i=1}^r\Phi_i^{n_i}$. Its reflex type is $(E_0'=\prod_{i=1}^r(E_i')^{n_i},\Phi_0'=\prod_{i=1}^r(\Phi_i')^{n_i})$. Write the composition field
	\[
	E'
	=
	E_1'E_2'\cdots E_r'\subset\mathbb{C}.
	\]
	So for each $i=1,\cdots,r$, we have a canonical map
	$\Nm_{E'/E_i'}\colon
	\Gal((E')^\mathrm{ab}/E')
	\to
	\Gal((E_i')^\mathrm{ab}/E_i')$.
	We define the product map
	$\widetilde{\rho}_{(E_0',\Phi_0')}
	=
	\prod_{i=1}^r\rho_{((E_i')^d,(\Phi_i')^d)}\circ\Nm_{E'/E_i'}$ such that the following diagram commutes
	\[
	\begin{tikzcd}
		\Gal((E')^\mathrm{ab}/E')
		\arrow[rr,"\prod_{i=1}^r\Nm_{E'/E_i'}"]
		\arrow[drr,"\widetilde{\rho}_{(E_0',\Phi_0')}"']
		&&
		\prod_{i=1}^r
		\Gal((E_i')^\mathrm{ab}/E_i')
		\arrow[d,"\prod_{i=1}^r\rho_{((E_i')^d,(\Phi_i')^d)}"]
		\\
		&&
		\prod_{i=1}^r
		(\overline{E_i^\times}\backslash\A_{E_i,f}^\times)^{n_i}
	\end{tikzcd}	
	\]
	Now the main theorems of complex multiplication gives
	\begin{proposition}
		The image of the map $\widetilde{\rho}_{(E_0',\Phi_0')}$ is contained in
		$\overline{Z_{\bfG_{\uA_0}}(\mathbb{Q})}
		\backslash Z_{\bfG_{\uA_0}}(\A_f)$. Moreover, if we fix an analytic parametrisation $\theta_0\colon\mathbb{C}^{\mathrm{dim}(A_0)}/\Phi_0(\mathfrak{a}_0)\simeq A_0(\mathbb{C})$ where $\mathfrak{a}_0$ is a lattice in $E_0$, then for any $\sigma\in\Gal((E')^\mathrm{ab}/E')$, we have the following commutative diagram
		\[
		\begin{tikzcd}
			E_0/\mathfrak{a}_0
			\arrow[r,"\theta_0\circ\Phi_0"]
			\arrow[d,"\widetilde{\rho}_{(E_0',\Phi_0')}(\sigma^{-1})"']
			&
			A_0(\overline{\mathbb{Q}})_{\mathrm{tor}}
			\arrow[d,"\sigma"]
			\\
			E_0/\rho_{(E_0',\Phi_0')}(\sigma^{-1})\mathfrak{a}_0
			\arrow[r,"\theta_\sigma\circ\Phi_0"]
			&
			\sigma(A_0)(\overline{\mathbb{Q}})_{\mathrm{tor}}
		\end{tikzcd}
		\]
	\end{proposition}

	Now we fix a finite non-empty set $\mathcal{S}$ of primes $\ell\neq p$ as in §\ref{Reduction of p-adic Hecke orbit of A} such that $\uA_0$ has supersingular good reduction at $v(\ell)$ for all $\ell\in\mathcal{S}$. For each $\ell\in\mathcal{S}$, we fix a non-empty finite set $\widetilde{\mathcal{T}}_\ell$ of elements $\sigma$ in $\Gal((E')^\mathrm{ab}/E')$ such that
	$\widetilde{\rho}_{(E_0',\Phi_0')}(\sigma)$ lies in  $\overline{Z_{\bfG_{\uA_0}}(\mathbb{Q})}\backslash Z_{\bfG_{\uA_0}}(\A_f)$. Then we put
	$\widetilde{\mathcal{T}}=\{\widetilde{\mathcal{T}}_\ell|\ell\in\mathcal{S}\}$.
	For each $\sigma\in\widetilde{\mathcal{T}}_\ell$, we fix an element $\sigma'\in Z_{\bfG_{\uA_0}}(\mathbb{Q}_p)$ such that the images of $\sigma$ and $\sigma'$ in
	\[
	Z_{\bfG_{\uA_0}}(\mathbb{Q})\backslash Z_{\bfG_{\uA_0}}(\A_f)/Z_{\bfG_{\uA_0}}(\A_f^{(p)})
	\simeq
	Z_{\bfG_{\uA_0}}(\mathbb{Q})\backslash Z_{\bfG_{\uA_0}}(\mathbb{Q}_p)
	\]
	are the same via the following natural maps
	\[
	\begin{tikzcd}
		&&
		\sigma'\in Z_{\bfG_{\uA_0}}(\mathbb{Q}_p)
		\arrow[d]
		\\
		\sigma\in\Gal(\widetilde{E}^\mathrm{ab}/\widetilde{E})
		\arrow[r,"\widetilde{\rho}_{(E_0',\Phi_0')}"]
		&
		\overline{Z_{\bfG_{\uA_0}}(\mathbb{Q})}
		\backslash
		Z_{\bfG_{\uA_0}}(\A_f)
		\arrow[r]
		&
		Z_{\bfG_{\uA_0}}(\mathbb{Q})\backslash Z_{\bfG_{\uA_0}}(\A_f)/Z_{\bfG_{\uA_0}}(\A_f^{(p)})
	\end{tikzcd}
	\]
	Then we put $\mathcal{T}_\ell
	:=\{\sigma'|\sigma\in\widetilde{\mathcal{T}}_\ell\}$, which is in bijection with
	$\widetilde{\mathcal{T}}_\ell$.
	The partition $\mathcal{T}_\ell=\bigsqcup_{i=1}^{r(\ell)}\mathcal{T}_{\ell,i}$ in (\ref{partition of T_l}) gives rise to a partition of $\widetilde{\mathcal{T}}_\ell$ under the above bijection
	\[
	\widetilde{\mathcal{T}}_\ell
	=
	\bigsqcup_{i=1}^{r(\ell)}\widetilde{\mathcal{T}}_{\ell,i}.
	\]

	We have thus the simultaneous reduction map
	\[
	\bfR_{\mathcal{S},\widetilde{\mathcal{T}}}
	\colon
	\mathcal{H}_p(\uA_0)
	\to
	\prod_{\ell\in\mathcal{S}}
	\prod_{\sigma\in\widetilde{\mathcal{T}}}
	\mathcal{H}_p(\uA_0),
	\quad
	\uA
	\mapsto
	((\sigma\uA)_{\ell})_{\ell\in\mathcal{S},\sigma\in\widetilde{\mathcal{T}}}.
	\]
	For a subset $\widetilde{\mathcal{T}}'$ of $\widetilde{\mathcal{T}}_\ell$, we have the twisted diagonal map $\widetilde{\Delta}^{\widetilde{\mathcal{T}}'}$, just like $\widetilde{\Delta}^{\mathcal{T}'}$ as in (\ref{twisted diagonal map}). Then Theorem \ref{R_{S,T} is an isomorphism} implies the following theorem, of which Theorem \ref{main theorem-3} is a special case:
	\begin{theorem}\label{surjectivity of Galois orbit}
		We have
		\[
		\mathrm{Im}(\bfR_{\mathcal{S},\widetilde{\mathcal{T}}})
		=
		\prod_{\ell\in\mathcal{S}}\prod_{i=1}^{r(\ell)}
		\widetilde{\Delta}^{\widetilde{\mathcal{T}}_{\ell,i}}(\mathcal{H}_p(\uA_{0}^{(\ell)})).
		\]
	\end{theorem}

	\section{Heegner points on Shimura curves and Mazur's conjecture}\label{Heegner points and Mazur's conjecture}
	In this section, we apply the previous results to Heegner points on Shimura curves. We first review some facts concerning Shimura curves, following closely \cite{Howard2004b} and \cite[§3]{CornutVatsal2007}.

	We fix a CM number field $E$ and let $F$ be the maximal totally real subfield of $E$. Write $d=[F:\Q]$. We fix a real embedding
	\[
	\xi\colon F\hookrightarrow\R,
	\]
	an ideal $\Nfrak$ of $\Ocal_F$ which is coprime to the relative discriminant $D_{E/F}$ of $E/F$. We write
	\[
	\epsilon_{E/F}\colon
	\A_F^\times/F^\times\to\{\pm1\}
	\]
	for the quadratic character attached to $E/F$. We assume the following weak Heegner hypothesis
	\begin{taggedtheorem}{(wH)}\label{weak Heegner hypothesis}
		$\epsilon_{E/F}(\Nfrak)=(-1)^{d-1}$.
	\end{taggedtheorem}
	Under \ref{weak Heegner hypothesis}, there is a unique quaternion algebra $B$ over $F$ which is ramified exactly at the primes $\Pfrak$ of $F$ that are inert in $E$ with $\mathrm{ord}_{\Pfrak}(\Nfrak)$ odd and at all the archimedean places of $F$ except $\xi$. We write $\Nfrak_B$ for the product of primes of $F$ where $B$ is ramified. We fix an isomorphism
	\[
	B\otimes_\Q\R
	\simeq
	\mathrm{M}_2(F_\xi)\oplus\Hbf^{d-1},
	\]
	where $\Hbf$ is the Hamiltonian quaternion algebra and $F_\xi$ is the completion of $F$ along $\xi$. In the following, we will exclude the case $B=\mathrm{M}_2(\Q)$ (which is already treated in \cite{Cornut2002}).\footnote{As the reader can see, our arguments below proving Theorems \ref{R_{S,T}^Gfrak} and \ref{Mazur's conjecture} covers also the case of modular curves. The only place where we use the exclusion of modular curves is in the definition of the map $ Sh_{U,F}\to J_{U,F}$, which does not affect our argument.} We view
	\[
	\G=B^\times
	\]
	as an algebraic group over $F$. We write $\Hfrak^\pm=\C\backslash\R$ for the union of upper and lower half planes. We let $\G(\A_{F,\infty})=\G(F\otimes_\Q\R)$ act on $\Hfrak^\pm$ via the projection to the $\xi$-component $\G(\A_{F,\infty})\to \G(F_\xi)\simeq\GL_2(\R)$. We then write $U_\infty$ for the stabilizer in $\G(\A_{F,\infty})$ of $i\in\Hfrak^\pm$. So we have an identification
	\[
	\G(\A_{F,\infty})/U_\infty\simeq\Hfrak^\pm.
	\]
	We have a section of the natural projection map $\G(\A_{F,\infty})\to\Hfrak^\pm$
	\[
	s\colon
	\Hfrak^\pm\to \G(\A_{F,\infty}),
	\]
	such that the $\xi$-component of $s(x+iy)$ is $\begin{pmatrix}
		y & x \\ 0 & 1
	\end{pmatrix}$, while the other components are trivial.
	
	By construction, for any place $v$ of $F$, if $B$ is ramified at $v$, then $E\otimes_FF_v$ is a field. Thus we have an embedding
	\[
	\kappa
	\colon
	E\hookrightarrow B.
	\]
	Moreover, there is a unique point $w(\kappa)\in\Hfrak^\pm$ that is fixed by $\kappa(x)$ for any $x\in E^\times$. We normalize $\kappa$ in such way that for any $x\in E^\times$, we have
	\[
	\kappa(x)
	\begin{pmatrix}
		w(\kappa) \\ 1
	\end{pmatrix}
	=
	x
	\begin{pmatrix}
		w(\kappa) \\ 1
	\end{pmatrix},
	\]
	where on the left $\kappa(x)$ is viewed as a matrix in $\GL_2(F_\xi)$ while on the right $x$ is viewed as a scalar in $\C^\times=(E\otimes_FF_\xi)^\times$ (\cite[§1.2]{Howard2004b}).

	We fix a maximal order $\Ocal_B$ of $B$ containing $\kappa(\Ocal_E)$ and write $R=\kappa(\Ocal_E)+\kappa(\Nfrak)\Ocal_B$. Then $R$ has reduced discriminant equal to $\Nfrak$. Under the identification $B(\A_{F,f})^\times=\G(\A_{F,f})$, we write $U$ for the image of $(R\otimes_\Z\widehat{\Z})^\times$ in $\G(\A_{F,f})$. The complex Shimura curve of level $U$ is given by
	\[
	Sh_U(\C)
	=
	\G(F)\backslash \G(\A_F)/Z_\G(\A)U_\infty U
	=
	\G(F)\backslash
	\left(
	\Hfrak^\pm\times \G(\A_{F,f})/Z_\G(\A_{F,f})U_f
	\right).
	\]
	It has a canonical model $ Sh_{U,F}$ over $F$, which is smooth, connected and projective. Since we exclude the case $B=\mathrm{M}_2(\Q)$, $ Sh_{U}(\C)$ is compact. It is well-known that $(\Res^F_\Q(\G),\Hfrak^\pm)$ is a Shimura datum of Hodge type (\cite{Mumford1969}). We write in the following $G=\Res^F_\Q(\G)$.

	For $(z,g)\in\Hfrak^\pm\times \G(\A_{F,f})$, we write $[(z,g)]$ for its image in $ Sh_{U}(\C)$. We let the normalizer of $U$ in $\G(\A_{F,f})$ act on $ Sh_{U}(\C)$ by the formula
	\[
	J(u)[(z,g)]:=[(z,gu^{-1})],
	\quad
	\forall
	u\in U \text{and } [(z,g)]\in Sh_U(\C).
	\]
	Write $\Nrm\colon \G(\A_F)\to\A_F^\times$ for the reduced norm. Then the set $\pi_0( Sh_{U}(\C))$ of connected components of $ Sh_{U}(\C)$ can be identified with $F^\times\backslash\A_F^\times/\Nrm(Z_\G(\A_F)U_\infty U)$. Write $F_U/F$ for the abelian extension such that under the Artin reciprocity, there is an isomorphism
	\[
	\Gal(F_U/F)\simeq
	F^\times\backslash\A_F^\times/\Nrm(Z_\G(\A_F)U_\infty U).
	\]
	We identify these two groups and define the map
	\[
	\GNrm\colon
	\G(\A_{F,f})\to\Gal(F_U/F),
	\quad
	g\mapsto\Nrm(g)^{-1}.
	\]
	We let $\Gal(F_U/F)$ act on $\pi_0( Sh_{U}(\C))$ via $\GNrm$ (\cite[§1.2]{Howard2004b})
	\[
	\begin{tikzcd}
		Sh_{U}(\C)
		\arrow[r,"J(u)"]
		\arrow[d,"\Nrm"']
		&
		Sh_{U}(\C)
		\arrow[d,"\Nrm"]
		\\
		\pi_0( Sh_{U}(\C))
		\arrow[r,"\GNrm(u)"]
		&
		\pi_0( Sh_{U}(\C))
	\end{tikzcd}
	\]

	We view $T=E^\times$ as an algebraic group over $F$, which is also a maximal torus of $\G$ via the embedding $\kappa$. Then we put
	\[
	\CM_E=T(F)\backslash \G(\A_{F,f})/Z_\G(\A_{F,f}),
	\]
	the set of CM points by $E$. We have a map
	\[
	\CM_E\to Sh_{U}(\C),
	\quad
	g\mapsto[(w(\kappa),g)].
	\]
	The image of this map is the set of CM points on $ Sh_{U}(\C)$. We let $\Gal(E^\ab/E)$ act on the set of CM points by the formula
	\begin{equation}\label{Gal(E^ab/E) acts on CM_E}
		[(w(\kappa),g)]^{[s,E]}:=[(w(\kappa),s\cdot g)],
		\quad
		\forall
		s\in T(\A_{F,f}).
	\end{equation}
	Here $[-,E]\colon T(F)\backslash T(\A_{F,f})\simeq\Gal(E^\ab/E)$ is the Artin map. It is easy to see that $Z_\G(\A_{F,f})=\A_{F,f}^\times$ acts trivially on any CM point.

	For any CM point $x=[(w(\kappa),g)]$, the endomorphism ring of $x$ is the preimage of $R\otimes_\Z\widehat{\Z}$ under the map $E\to B\otimes_\Z\widehat{\Z}$, which sends $h$ to $g^{-1}hg$. This endomorphism ring is an order of $E$ of the form $\Ocal_{\Cfrak_x}=\Ocal_F+\Cfrak_x\Ocal_E$ for some integral ideal $\Cfrak_x\subset\Ocal_F$ (the conductor of $x$). We write $T[\Cfrak_x]=\kappa(\widehat{\Ocal}_{\Cfrak_x}^\times)\subset T(\A_{F,f})$. We write $E[\Cfrak]$ for the ring class field of $E$, the abelian extension of $E$ corresponding to $T[\Cfrak]Z_\G(\A_{F,f})$ via class field theory. This is a Galois extension of $F$ and is the field of definition of $x$.

	For each prime $\lfrak$ of $\Ocal_F$ prime to $\Nfrak D_{E/F}$, we fix an isomorphism
	\begin{equation}\label{R_ell and M_2(O_{F,ell})}
		R_\lfrak\simeq\mathrm{M}_2(\Ocal_{F,\lfrak})
	\end{equation}
	such that $\kappa(\Ocal_{E,\lfrak})$ is identified with the set of diagonal matrices $\begin{pmatrix}
		x & 0 \\ 0 & y
	\end{pmatrix}$, resp. $\begin{pmatrix}
		x & yu \\ y & x
	\end{pmatrix}$ for $x,y\in\Ocal_{F,\lfrak}$ if $\lfrak$ splits in $E$, resp. $\lfrak$ is inert in $E$. Here $u$ is a fixed non-square element in $\Ocal_{F,\lfrak}^\times$.
	This extends to an isomorphism $B_\lfrak\simeq\mathrm{M}_2(F_\lfrak)$. Fix a uniformiser $\varpi_\lfrak$ of $F_\lfrak$ and for each integer $k\ge0$, define an element $h[\lfrak^k]$ in $B_\lfrak$ whose image in $\mathrm{M}_2(F_\lfrak)$ is given by
	\[
	\begin{cases*}
		\begin{pmatrix}
			\varpi_\lfrak^k & 1 \\ 0 & 1
		\end{pmatrix},
		&
		if $\lfrak$ splits in $E$;
		\\
		\begin{pmatrix}
			\varpi_\lfrak^k & 0 \\ 0 & 1
		\end{pmatrix},
		&
		if $\lfrak$ inert in $E$.
	\end{cases*}
	\]
	We view $h[\lfrak^k]$ as elements in $\G(\A_{F,f})$ via the embedding $\G(F_\lfrak)\hookrightarrow \G(\A_{F,f})$. For any integral ideal $\Pfrak=\prod_{i=1}^r\lfrak_i^{k_i}$ of $\Ocal_F$ prime to $\Nfrak D_{E/F}$, we define $h[\Pfrak]:=\prod_{i=1}^rh[\lfrak_i^{k_i}]$. We view $h[\Pfrak]$ also as elements in $\CM_E$ and write $x[\Pfrak]$ for its image in $ Sh_{U}(\C)$.
	It is well known that these $x[\Pfrak]$ are defined over $E^\ab$. Moreover, the action of $\Gal(E^\ab/E)$ on these $x[\Pfrak]$ agrees with the action in (\ref{Gal(E^ab/E) acts on CM_E}) (\cite[Theorem 1.2.2]{Howard2004b}).

	We decompose $ Sh_{U,F}\times_FF_U$ into geometric components $ Sh_{U,F}\times_FF_U=\bigsqcup_{i} Sh_{U,F,i}$ with each $ Sh_{U,F,i}$ geometrically irreducible and for any extension $F'/F$ such that $ Sh_{U,F}(F')\neq\emptyset$ must contain $F_U$. We define
	\[
	J_{U,F}:=\Res^{F_U}_F(\mathrm{Jac}( Sh_{U,F,i_0}))
	\]
	for some fixed component $ Sh_{U,F,i_0}$. Here $\mathrm{Jac}(Sh_{U,F,i_0})$ is the Jacobian of $Sh_{U,F,i_0}$. Then $J_{U,F}$ is an abelian variety defined over $F$, which has good reduction away from $\Nfrak$. There is a unique element $\hfrak\in\mathrm{Pic}( Sh_{U,F})$ (up to a constant multiple) whose degree on each geometric component $ Sh_{U,F,i}$ is constant and which satisfies $T_\Mfrak\hfrak=\degrm(T_\Mfrak)\hfrak$ for every Hecke operator $T_\Mfrak$ with $\Mfrak$ prime to $\Nfrak D_{E/F}$. This is called the Hodge class and on each geometric component $\hfrak$ is just the canonical divisor. We write $\hfrak_i$ for the restriction of $\hfrak$ to the component $ Sh_{U,F,i}$, then we have a unique morphism defined over $F$
	\[
	Sh_{U,F}\to J_{U,F},
	\]
	which, on complex points, takes $x_i\in  Sh_{U,F}(\C)$ to the divisor $dx_i-\hfrak_i\in J_{U,F}(\C)$.

	Let $A/F$ be an abelian variety with a surjective morphism of abelian varieties $\alpha\colon J_{U,F}\to A$. Composed with the map $ Sh_{U,F}\to J_{U,F}$, we have a morphism over $F$
	\[
	\pi\colon Sh_{U,F}\to J_{U,F}\to A.
	\]
	Since $ Sh_{U,F}$ is complete and connected, the morphism $\pi$ is either finite or constant, the latter of which is impossible since $\alpha$ is surjective. Thus $\pi$ is a finite morphism (\emph{cf.} \cite[Lemma 3.9]{CornutVatsal2007}).

	Now we follow the notations in §\ref{Shimura varieties of Hodge type}.
	We fix a prime $p$ that is prime to $\Nfrak D_{E/F}$ where $D_{E/F}$ is the relative discriminant of $E/F$. In particular, we have
	\begin{equation}\label{G^1(Q_p) decomposition}
		G^1(\Q_p)/Z_{G^1}(\Q_p)
		=
		\G^1(F_\pfrak)/Z_{\G^1}(F_\pfrak)\simeq\Pbf\mathrm{SL}_2(F_p)
		=
		\prod_{\pfrak|p}\Pbf\mathrm{SL}_2(F_\pfrak),
	\end{equation}
	which is a product of simple, non-commutative groups. Here $\pfrak$ runs through the primes of $F$ over $p$. We fix a CM point $x_0\in Sh_{U,F}(\overline{\Q})$ and a non-empty finite set $\Scal$ of primes $\lfrak$ different from $p$ such that $x_0$ has supersingular good reduction at all places of $\overline{\Q}$ over $\lfrak$. For each $\lfrak\in\Scal$, we fix a non-empty finite set $\Tcal_\lfrak$ of elements in $Z_{I_{x_0}}(\A_{f})=I_{x_0}(\A_{f})=
	(\Ocal_{\Cfrak_x}\otimes_\Z\widehat{\Q})^\times$. Then for the following simultaneous reduction map
	\[
	\Rbf_{\Scal,\Tcal}^G\colon
	\Hcal_p^G(\Acal_{x_0})\to
	\prod_{\lfrak\in\Scal}\prod_{\sigma\in\Tcal_\lfrak}\Hcal_p^{G}(\Acal_{x_0^{(\lfrak)}}),
	\]
	we apply Theorem \ref{R_{S,T}, not simple} and get
	\begin{theorem}\label{R_{S,T}^Gfrak}
		The image $\Rbf_{\Scal,\Tcal}^G(\Hcal_p^G(\Acal_{x_0}))$ is equal to
		\[
		\prod_{\lfrak\in\Scal}\prod_{i=1}^{r(\lfrak)}
		\widetilde{\Delta}^{\Tcal_{\lfrak,i}}(\Hcal_p^G(\Acal_{x_0^{(\lfrak)}})).
		\]
	\end{theorem}
	\begin{proof}
		We need to check that the assumptions in Theorem \ref{R_{S,T}, not simple} are satisfied:
		\begin{enumerate}
			\item 
			The condition (\ref{condition on similitude of I_x}) holds because $R$ is a maximal order in $B$ of reduced discriminant $\Nfrak$ prime to $p$;

			\item
			We have $G^1=\Res^F_\Q(\G^1)$, which is simply connected;

			\item 
			We enumerate the primes of $F$ over $p$ as $\pfrak_1,\cdots,\pfrak_t$ and write $G^{(j)}=\Pbf\mathrm{SL}_2(F_{\pfrak_j})$ for $j=1,\cdots,t$. Then we have $\Pbf \mathrm{SL}_2(F_p)=\prod_{j=1}^tG^{(j)}$ (\ref{G^1(Q_p) decomposition});

			\item 
			We prove that for distinct $\lfrak,\lfrak'\in\Scal$, $\Gamma_p(\lfrak)$ and $\Gamma_p(\lfrak')$ are not commensurable. If this is not the case, then we have
			\[
			I_{x_0^{(\lfrak)}}(\Q)
			\subset
			\Ccal_{G(\Q_p)}(\Gamma_p(\lfrak))
			=
			\Ccal_{G(\Q_p)}(\Gamma_p(\lfrak'))
			\supset
			I_{x_0^{(\lfrak')}}(\Q).
			\]
			Write $B(\lfrak)$ for the quaternion algebra over $F$ such that $I_{x_0^{(\lfrak)}/\Q}=\mathrm{Res}^F_\Q(B(\lfrak)^\times)$ as group schemes over $\Q$ (similarly for $B(\lfrak')$), then we would have
			\[
			B(\lfrak)=B(\lfrak'),
			\]
			which is impossible as they have different ramified primes.

			\item 
			We need to show that $\Gamma_p(\lfrak)G^{(j)}$ is dense in $G:=\Pbf\mathrm{SL}_2(F_p)$. We need to show that for any compact open subgroup $U'_p\subset G$, $\Gamma_p(\lfrak)G^{(j)}U'_p=G$. By strong approximation for $I_{x_0^{(\lfrak)}}^1$, we know that both $I_{x_0^{(\lfrak)}}(\Q)G^{(j)}$ and  $I_{x_0^{(\lfrak)}}(\Q)G$ is dense in $I_{x_0^{(\lfrak)}}(\widehat{\Q})$. In particular, take a compact open subgroup $V=\prod_qV_q$ of $I_{x_0^{(\lfrak)}}(\widehat{\Q})$ such that $V^{(p)}=U^{(p)}$ and $V_p=U_p'$. Then we have
			\[
			I_{x_0^{(\lfrak)}}(\Q)G^{(j)}V
			=I_{x_0^{(\lfrak)}}(\widehat{\Q})
			=I_{x_0^{(\lfrak)}}(\Q)GV.
			\]
			This gives immediately the following, which is what we need to prove:
			\[
			\left(I_{x_0^{(\lfrak)}}(\Q)\bigcap V^{(p)}\right)G^{(j)}V_p=G.
			\]

		\end{enumerate}
	\end{proof}

	We assume that the CM point $x_0=[(z,g)]\in Sh_{U}(\overline{\Q})$ has trivial prime-to-$\pfrak$-component $g^\pfrak=1$. Using the description of the action of $\Gal(E^\ab/E)$ on the CM points in (\ref{Gal(E^ab/E) acts on CM_E}) and the identification in (\ref{R_ell and M_2(O_{F,ell})}) for the case $\lfrak=\pfrak$, one has inclusions
	\begin{equation}\label{Heegner point, Hecke orbit, Galois orbit of Heegne point}
		\{x[\pfrak^k]|k\in\N\}
		\subset
		\Hcal_p^G(\Acal_{x_0})
		\subset
		\{\sigma(x[\pfrak^k])|k\in\N,\,\sigma\in\Gal(E^\ab/E)\}.
	\end{equation}
	For each prime $\lfrak\in\Scal$ over $\ell$, we write $ Sh_{U}^\ssrm(\F_{v(\lfrak)})$ for the supersingular locus of $ Sh_{U}(\F_{v(\lfrak)})$ where $\F_{v(\lfrak)}=\overline{\F}_\ell$ is the residual field of $\Ocal_{\overline{\Q}}$ at the place $v(\lfrak)$. Then we have bijections (\cite[Corollary 3.13 \& §3.2.5]{CornutVatsal2007})
	\[
	Sh_{U}^\ssrm(\F_{v(\lfrak)})
	\simeq
	(B(\lfrak)^\times\bigcap U^\pfrak)
	\backslash
	(B(\lfrak)\otimes_\Q\Q_p)^\times/U_\pfrak
	\simeq
	(I_{x_0^{(\lfrak)}}(\Q)\bigcap U^p)
	\backslash
	I_{x_0^{(\lfrak)}}(\Q_p)/U_p
	\simeq
	\Hcal_p^G(\Acal_{x_0}^{(\lfrak)}).
	\]
	Using Eichler Mass Formula, we know that as $\ell\to\infty$ ($\ell$ is the prime under $\lfrak$), the cardinal $\# Sh_{U}^\ssrm(\F_{v(\lfrak)})\to\infty$ (see also \cite[§4.6]{CornutVatsal2007}). We write $t_0$ for the upper bound for the cardinals of fiber of the map $\pi\colon Sh_{U}\to A$, which is finite as $\pi$ is a finite morphism. By \cite[Lemma 4.1]{Cornut2002}, we know that $A(E[\pfrak^\infty])_\mathrm{tors}$ is finite, say, of cardinal $t_1$. We assume in the following that for each $\lfrak\in\Scal$,
	\[
	\# Sh_{U}^\ssrm(\F_{v(\lfrak)})>t_1t_2.
	\]
	Then there exist two points $x_\lfrak',x_\lfrak''\in Sh_{U}^\ssrm(\F_{v(\lfrak)})$ such that $\pi(x_\lfrak')-\pi(x_\lfrak'')$ is non-zero and does not lie in the image of $A(E[\pfrak^\infty])_\mathrm{tors}$ under the map $A(E[\pfrak^\infty])\to A(\F_{v(\lfrak)})$.

	For a prime $\ell'$ and a place $\lfrak'$ of $F$ over $\ell'$, it is well-known that any principally polarized supersingular abelian variety of dimension $>1$ over $\overline{\F}_{\ell'}$ is actually defined over $\F_{(\ell')^2}$ (\cite[Theorem 1]{IbukiyamaKatsura1994}). We write $\F_{\lfrak'}'$ for the compositum field of the residual field $\F_{\lfrak'}$ of $\Ocal_F$ at $\lfrak'$ and the field $\F_{(\ell')^2}$ (both viewed as subfields of $\overline{\F}_{\ell'}$). It follows that each point in $ Sh_{U}^\ssrm(\overline{\F}_{\ell'})$ is defined over $\F_{\lfrak'}'$.
	We then write $J_U^\ssrm(\F_{\lfrak'}')$ for the subgroup of $J_U(\F_{\lfrak'}')$ generated by $x-y$ with $x,y\in Sh_{U}^\ssrm(\F_{\lfrak'}')$.

	We write the torsion subgroup
	\[
	G_0=\Gal(E[\pfrak^\infty]/E)_\mathrm{tors},
	\]
	which is a finite abelian group. Then we have $G_0=\Gal(E[\pfrak^\infty]/H_\infty)$ where $H_\infty/E$ is the maximal anti-cyclotomic $\Z_p^r$-extension where $r=[F_\pfrak:\Q_p]$. For any ring $R$, we define an element $e=\sum_gg$ in the group ring $R[G_0]$.

	Now we state and prove Mazur's conjecture (\cite[p.203]{Mazur1983}):
	\begin{theorem}\label{Mazur's conjecture}
		Assume \ref{weak Heegner hypothesis}, then there is $n\ge0$ such that $\Trrm_{E[\pfrak^\infty]/H_\infty}(y[\pfrak^n])\notin A(H_\infty)_\mathrm{tors}$.
	\end{theorem}
	\begin{proof}
		Let $\Scal$ be as before and consist of only one prime $\lfrak$. We write $\Tcal_\lfrak$ to be a set of representatives in $I_{x_0}(\A_{f})=\A_{E,f}^\times$ of $G_0$ under the Artin reciprocity isomorphism.

		We argue by contradiction. Suppose that $\Trrm_{E[\pfrak^\infty]/H_\infty}(y[\pfrak^n])$ are torsion for all $n\ge0$, and thus their orders all divide $t_1$. By the inclusions (\ref{Heegner point, Hecke orbit, Galois orbit of Heegne point}), we know that for any $\Acal_x\in\Hcal_p^G(\Acal_{x_0})$, its images $\pi(\Acal_x)\in A(E[\pfrak^\infty])$ and $\pi(\Acal_{x^{(\lfrak)}})\in A(\F_{\lfrak}')$ as well as their Galois conjugates all have orders dividing $t_1$.

		For $x,y\in Sh_{U}^\ssrm(\F_{v(\lfrak)})$, by Theorem \ref{R_{S,T}^Gfrak}, we can choose $\Acal_a,\Acal_b\in\Hcal_p^G(\Acal_{x_0})$ such that all the components of $\Rbf_{\Scal,\Tcal}(\Acal_a)$ are equal to those of $\Rbf_{\Scal,\Tcal}(\Acal_b)$, except at the components in  $\Tcal_{\lfrak,i_0}\subset\Tcal_\lfrak$ for some $i_0\in r(\lfrak)$, where
		\[
		\Rbf_\lfrak(\sigma\Acal_a)=x
		\quad\text{and}\quad
		\Rbf_\lfrak(\sigma\Acal_b)=y,
		\quad
		\forall\sigma\in\Tcal_{\lfrak,i_0}.
		\]
		In particular, we have
		\[
		\Rbf_\lfrak(e(\pi(\Acal_a)))-\Rbf_\lfrak(e(\pi(\Acal_b)))
		=
		\left(0,\cdots,0,
		\sum_{\sigma\in\Tcal_{\lfrak,i_0}}
		\left(\Acal_{(\sigma a)^{(\lfrak)}}-\Acal_{(\sigma b)^{(\lfrak)}}\right),
		0,\cdots,0
		\right).
		\]
		Here on the right hand side, all the components are zero except the $\lfrak$-th component. Thus, fixing one element $\sigma_0\in\Tcal_{\lfrak,i_0}$, we have
		\[
		\sum_{\sigma\in\Tcal_{\lfrak,i_0}}
		\left(\Acal_{(\sigma a)^{(\lfrak)}}-\Acal_{(\sigma b)^{(\lfrak)}}\right)
		=\#(\Tcal_{\lfrak,i_0})(\pi(x)-\pi(y)).
		\]

		We set $L=E(A[q])$, which is a finite Galois extension of $F$. We fix an embedding $L\hookrightarrow\C$ and write $c\in\Gal(L/F)$ for the complex conjugation. Then we write $\Scal_c$ for the set of primes $\lfrak$ of $F$ prime to $2\Nfrak D_{E/F}t_1$ such that the Frobenius $\mathrm{Frob}_\lfrak$ in $\Gal(L/F)$ is equal to $c$. By Chebotarev's density theorem, we know $\#(\Scal_c)=\infty$. For each $\lfrak\in\Scal_c$ over a rational prime $\ell$, fix a prime $\Lfrak$ of $L$ over $\lfrak$ such that $\mathrm{Frob}_\Lfrak=c$ in $\Gal(L/F)$. Fix then a place $v(\lfrak)'$ of $E[\pfrak^\infty](A[t_1])$ over $\Lfrak$ and $v(\lfrak)''$ its restriction to $E[\pfrak^\infty]$. Then $\lfrak$ is inert in $E$, $\F_{v(\lfrak)''}\supset\F_{\lfrak}'$ and for each prime $q\mid t_1$, we have
		\[
		\dim_{\F_q}(A(\F_{v(\lfrak)''})\otimes\F_q)
		=
		2\dim(A).
		\]
		In particular,
		\[
		\# A([\F_{v(\lfrak)''}])[t_1]\leq t_1^{2\dim(A)}.
		\]

		We choose $\lfrak\in\Scal_c$ such that
		\[
		\# Sh_{U}^\ssrm(\F_{v(\lfrak)})>(t_1\cdot\#G_0)^{2\dim(A)}\cdot t_2.
		\]
		Then there exists $x,y\in Sh_{U}^\ssrm(\F_{v(\lfrak)})$ such that $\pi(\Acal_x)-\pi(\Acal_y)\in A(\F_{\lfrak}')\subset A(\F_{v(\lfrak)''})$ has order not dividing $t_1\cdot\#G_0$ (note that $\#\Tcal_{\lfrak,i_0}$ divides $\#G_0$). In particular, $\Rbf_\lfrak(e(\pi(\Acal_a)))-\Rbf_\lfrak(e(\pi(\Acal_b)))$ has order not dividing $t_1$.
		This contradicts the fact that $\pi(\Acal_{a^{(\lfrak)}})=\pi(\Acal_x)$ and $\pi(\Acal_{b^{(\lfrak)}})=\pi(\Acal_y)$ have orders dividing $t_1$, which thus finishes the proof of the theorem.
	\end{proof}

\end{document}